\documentclass[11pt,a4paper]{amsart}
\usepackage{a4wide}
\usepackage{comment,mathrsfs, amssymb,amsmath,url,amsthm}

\title{The Wonderland of Reflections}
\date{\today}
\author%[L.~Barto]
{Libor Barto}
	\address{Department of Algebra, MFF UK, Sokolovska 83, 186 00 Praha 8, Czech Republic}
	\email{libor.barto@gmail.com}
	\urladdr{http://www.karlin.mff.cuni.cz/~barto/}
\author%[J.~Opr\v{s}al]
{Jakub Opr\v{s}al}
	\address{Department of Algebra, MFF UK, Sokolovska 83, 186 00 Praha 8, Czech Republic}
	\email{oprsal@karlin.mff.cuni.cz}
	\urladdr{http://www.karlin.mff.cuni.cz/~oprsal/}
\author%[M.~Pinsker]
{Michael Pinsker}
	\address{Department of Algebra, MFF UK, Sokolovska 83, 186 00 Praha 8, Czech Republic}    
	\email{marula@gmx.at}
    \urladdr{http://dmg.tuwien.ac.at/pinsker/}
\thanks{Libor Barto and Jakub Opr\v sal were supported by the Grant Agency of the Czech Republic, grant GA\v CR
13-01832S. Michael Pinsker has been funded through project  P27600 of the  Austrian Science Fund (FWF)}

\theoremstyle{plain}

    \newtheorem{thm}{Theorem}[section]
    \newtheorem{theorem}[thm]{Theorem}

    \newtheorem{lemma}[thm]{Lemma}

    \newtheorem{prop}[thm]{Proposition}
    \newtheorem{proposition}[thm]{Proposition}

    \newtheorem{cor}[thm]{Corollary}
    \newtheorem{corollary}[thm]{Corollary}

    \newtheorem{prob}[thm]{Problem}
    \newtheorem{conj}[thm]{Conjecture}
		\newtheorem{conjs}[thm]{Conjectures}

\theoremstyle{definition}

    \newtheorem{defn}[thm]{Definition}
    \newtheorem{definition}[thm]{Definition}

    \newtheorem{example}[thm]{Example}

\theoremstyle{remark}

\newcommand{\transfer}{reflection}
\newcommand{\transfers}{reflections}

\newcommand{\alg}[1]{\mathbf{#1}}
\newcommand{\clone}[1]{\mathscr{#1}}
\newcommand{\relstr}[1]{\mathbb{#1}}
\newcommand{\class}[1]{\mathcal{#1}}
\newcommand{\var}[1]{\mathcal{#1}}
\newcommand{\varclo}[1]{\mathrm{clo}(#1)}

\newcommand{\algA}{\alg{A}}

\newcommand{\cloA}{\clone{A}}
\newcommand{\cloB}{\clone{B}}
\newcommand{\cloC}{\clone{C}}
\newcommand{\relA}{\relstr{A}}
\newcommand{\relB}{\relstr{B}}

\DeclareMathOperator{\PpInt}{\mathsf Pp-int}
\DeclareMathOperator{\PpPower}{\mathsf Ppp}
\DeclareMathOperator{\HomEq}{\mathsf He}

\DeclareMathOperator{\EHSP}{\mathsf{EHSP}}
\DeclareMathOperator{\ETraP}{\mathsf{ERP}}
\DeclareMathOperator{\ETra}{\mathsf{ER}}
\DeclareMathOperator{\RPfin}{\mathsf{RP}_{fin}}
\DeclareMathOperator{\RP}{\mathsf{RP}}
\DeclareMathOperator{\RR}{\mathsf{RR}}
\DeclareMathOperator{\SP}{\mathsf{SP}}
\DeclareMathOperator{\PR}{\mathsf{PR}}
\DeclareMathOperator{\EREPfin}{\mathsf{EREP}_{fin}}
\DeclareMathOperator{\EPfin}{\mathsf{EP}_{fin}}
\DeclareMathOperator{\ETraPfin}{\mathsf{ERP}_{fin}}
\DeclareMathOperator{\ERetP}{\mathsf{ER}_{ret}\mathsf P}
\DeclareMathOperator{\EHSPfin}{\mathsf{EHSP}_{fin}}
\DeclareMathOperator{\HSP}{\mathsf{HSP}}
\DeclareMathOperator{\HSPfin}{\mathsf{HSP}_{fin}}
\DeclareMathOperator{\SSS}{\mathsf S}
\DeclareMathOperator{\EEE}{\mathsf E}
\DeclareMathOperator{\PPP}{\mathsf P}
\DeclareMathOperator{\PPPfin}{\mathsf P_{fin}}
\DeclareMathOperator{\HHH}{\mathsf H}
\DeclareMathOperator{\Tra}{\mathsf R}
\DeclareMathOperator{\Ret}{\mathsf R_{ret}}
\DeclareMathOperator{\Exp}{\mathsf E}
\DeclareMathOperator{\Pfin}{\mathsf P_{fin}}

\DeclareMathOperator{\Aut}{Aut}

\DeclareMathOperator{\End}{End}
\DeclareMathOperator{\Pol}{Pol}

\DeclareMathOperator{\CSP}{CSP}

\newcommand{\ignore}[1]{}

\newcommand{\To}{\rightarrow}

\begin{document}

\begin{abstract}
A fundamental fact for the algebraic theory of constraint satisfaction problems (CSPs) over a fixed template is that pp-interpretations between at most countable $\omega$-categorical relational structures have two algebraic counterparts for their polymorphism clones: a semantic one via the standard algebraic operators $\HHH$, $\SSS$, $\PPP$, and a syntactic one via clone homomorphisms (capturing identities). 
We provide a similar characterization which incorporates \emph{all} relational constructions relevant for CSPs, that is, homomorphic equivalence and adding singletons to cores in addition to pp-interpretations. For the semantic part we introduce a new construction, called \transfer, and for the syntactic part we find an appropriate weakening of clone homomorphisms, called h1 clone homomorphisms (capturing identities of height $1$). 

As a consequence, the complexity of the CSP of an at most countable $\omega$-categorical structure depends only on the identities of height $1$ satisfied in its polymorphism clone as well as the natural uniformity thereon. This allows us in turn to formulate a new elegant dichotomy conjecture for the CSPs of reducts of finitely bounded homogeneous structures.

Finally, we reveal a close connection between h1 clone homomorphisms and the notion of compatibility with projections used in the study of the lattice of interpretability types of varieties. 
%
%We introduce a new weaker algebraic abstraction of function clones than the one classically used in universal algebra; this corresponds to a new weaker notion of homomorphism between function clones. Via a variant of Birkhoff's $\HSP$ theorem for this new notion, we characterize when a function clone can be obtained from another one via arbitrary powers and so-called \transfers. Taking a topological perspective in addition to the algebraic one, we obtain a similar characterization for finite powers and \transfers.
%
%We show that these weaker notions of homomorphism cover all known CSP reductions. In particular, whether a given CSP of a finite structure reduces to the CSP of an $\omega$-categorical structure only depends on the topology of the polymorphism clone of the latter as well as the equations of height 1  satisfied therein. This allows us to formulate a new dichotomy conjecture for the CSPs of reducts of finitely bounded homogeneous structures.
%
%We introduce a construction on algebras, called a \transfer, which captures (up to pp-definability) homomorphic equivalence of the corresponding relational structures. 
\end{abstract}
\maketitle

\section{Introduction and Main Results}

The motivation for this work is to resolve some unsatisfactory aspects in the fundamentals of the theory of fixed-template constraint satisfaction problems (CSPs). The CSP over a relational structure $\relA$ in a finite language, denoted $\CSP(\relA)$, is the decision problem which asks whether a given primitive positive (pp-) sentence over $\relA$ is true. The focus of the theoretical research on such problems is to understand how the complexity of $\CSP(\relA)$, be it computational or descriptive, depends on the structure $\relA$.

We start by briefly reviewing and discussing the basics of the theory, first for
structures $\relA$ with a finite universe, and then for those with an infinite
one. The pioneering papers for finite structures $\relA$ are \cite{FV98} and
\cite{JBK}, and our presentation is close to the recent
survey~\cite{BSL:9956673}. For infinite structures $\relA$, a detailed account
of the current state of the theory can be found in~\cite{Bodirsky-HDR}, and a
compact introduction in~\cite{Pin15}. For the sake of compactness, we will
define standard notions only after this introduction, in
Section~\ref{sect:prelims}.

\subsection{The finite case} For finite relational structures $\relA$ and $\relB$ there 
are three general reductions which are used to compare the complexity of their CSPs. 
Namely, we know that  $\CSP(\relstr{B})$ is at most as hard as $\CSP(\relstr{A})$ if
\begin{itemize}
\item[(a)] $\relB$ is pp-interpretable in $\relA$, or
\item[(b)] $\relB$ is homomorphically equivalent to $\relA$, or
\item[(c)] $\relA$ is a core and $\relB$ is obtained from $\relA$ by adding a singleton unary relation.
\end{itemize} 
Item (a) has two algebraic counterparts. The semantic one, item (ii) in Theorem~\ref{thm:old_finite} below, follows from the well-known Galois correspondence between relational clones and function clones (see for example~\cite{Szendrei}), while the syntactic one, item (iii) in the same theorem, follows from the Birkhoff's HSP theorem~\cite{Bir-On-the-structure}.

\begin{theorem} \label{thm:old_finite}
Let $\relA$, $\relB$ be finite relational structures and $\cloA$, $\cloB$ their polymorphism clones.
Then the following are equivalent.
\begin{itemize}
\item[(i)] $\relB$ is pp-interpretable in $\relA$.
\item[(ii)] $\cloB \in \EHSPfin \cloA $, or equivalently, $\cloB \in \EHSP \cloA$;  here, $\EEE$ denotes the expansion operator.
\item[(iii)] There exists a clone homomorphism from $\cloA$ into $\cloB$, i.e., a  mapping $\cloA \to \cloB$ preserving identities. 
\end{itemize}
\end{theorem}

One unsatisfactory feature of this theorem is that it does not cover the other
two reductions~(b) and~({c}), in particular the easiest reduction to
homomorphically equivalent structures. The way around this fact is, usually, to
assume that those reductions have already been applied. This is the same as
saying that we can ``without loss of generality'' assume that structures are
cores containing all unary singleton relations, or equivalently that their
polymorphism clones are idempotent, and then only use  reductions by
pp-interpretations. However, this causes slightly awkward formulations, e.g., of
the conjectured condition for polynomial solvability~\cite{JBK} (also see
Conjecture~\ref{conj:old_finite} below) or of the condition for expressibility
in Datalog~\cite{LZ07,BK14}. Even worse, it results in a loss of power:
Example~\ref{ex:hepp} shows that there are cores $\relA$ and $\relB$ such that
$\relB$ is not pp-interpretable in $\relA$, but $\relB$ is homomorphically
equivalent to a structure which is pp-interpretable in $\relA$.  These
considerations bring up the following question: is there a variant of
Theorem~\ref{thm:old_finite} that covers all three reductions~(a), (b), and
({c})?

Another question concerns item (iii), which implies that the complexity of $\CSP(\relA)$ depends only on identities satisfied by operations in the polymorphism clone $\cloA$. However, the polymorphism clones of homomorphically equivalent structures need not necessarily satisfy the same identities,  with the exception of height $1$ identities. 
%Moreover, most of the classical conditions Maltsev conditions are linear. 
Naturally, the question arises: Is it possible to prove that the complexity of the CSP of a structure depends only on the height $1$ identities that hold in its polymorphism clone?

Finally, and related to the preceding question, the finite tractability conjecture~\cite{JBK} states that, assuming P${}\neq{}$NP, the CSP of a finite core $\relB$ is NP-hard if and only if  the idempotent reduct of its polymorphism clone $\cloB$ does not satisfy any non-trivial identities; here, we say that identities are non-trivial if they are not satisfiable in the clone of projections on a two-element set, which we denote by $\mathbf 1$. For general finite structures, the conjecture can be formulated as follows.  

\begin{conj} \label{conj:old_finite}
Let $\relA$ be a finite relational structure and let $\relB$ be its idempotent core, i.e., its core expanded by all singleton unary relations.
Then one of the following holds.
\begin{itemize}
\item The polymorphism clone $\cloB$ of $\relB$ maps homomorphically to $\mathbf{1}$ (and consequently $\CSP(\relA)$ is NP-complete).
\item $\CSP(\relA)$ is solvable in polynomial-time.
\end{itemize}
\end{conj}

\noindent
Is it possible to find a criterion on the structure of the polymorphism clone $\cloA$ of $\relA$, rather than $\cloB$, which divides NP-hard from polynomial-time solvable CSPs for all finite structures $\relA$, without the necessity to consider their cores? 

It turns out that Theorem~\ref{thm:old_finite} can be generalized to answer all three questions in positive. First, we observe that $\relB$ can obtained from $\relA$ by using any number of the constructions (a),~(b),~({c}) if and only if $\relB$ is homomorphically equivalent to a pp-power of $\relA$, where pp-power is a simplified version of pp-interpretation which we are going to define. We then introduce a simple algebraic construction, the \transfer{}%
\footnote{Also known as the \emph{double shrink}.}, which is in a sense an algebraic counterpart to homomorphic equivalence. This gives us a suitable generalization of item (ii) in Theorem~\ref{thm:old_finite}. Finally, we provide an analogue of Birkhoff's HSP theorem for classes of algebras described by height $1$ identities, by which we obtain syntactic characterization corresponding to item~(iii). Altogether, we get the following.

\begin{theorem} \label{thm:new_finite}
Let $\relA$, $\relB$ be finite relational structures and $\cloA$, $\cloB$ their polymorphism clones.
Then the following are equivalent.
\begin{itemize}
\item[(i)] $\relB$ is homomorphically equivalent to a pp-power of $\relA$, or equivalently, $\relB$ can be obtained from $\relA$ by a finite number of constructions among (a),~(b),~({c}).
\item[(ii)] $\cloB \in \ETraPfin \cloA$, or equivalently, $\cloB \in \ETraP \cloA$; here, $\Tra$ denotes the new operator of taking \transfers{}.
\item[(iii)] There exists an h1 clone homomorphism from $\cloA$ into $\cloB$, i.e., a  mapping $\cloA \to \cloB$ preserving identities of height $1$. 
\end{itemize}
\end{theorem}

This allows us to rephrase the conjectured sufficient condition for polynomial solvability. In the following theorem, items (i) -- (iv) are equivalent by~\cite{T77,BK12,Sig10,KMM14}, the primed items are new, core-free versions (see Section~\ref{sec:wrapup} for more details).

\begin{theorem} \label{thm:equiv_conditions}
Let $\relA$ be a finite relational structure, let $\relB$ be its idempotent core, and let $\cloA$, $\cloB$ be the polymorphism clones of $\relA$, $\relB$. Then the following are equivalent.
\begin{itemize}
\item[(i)] there is no clone homomorphism from $\cloB$ to $\mathbf{1}$.
\item[(ii)] there is no h1 clone homomorphism from $\cloB$ to $\mathbf{1}$.
\item[(ii')] there is no h1 clone homomorphism from $\cloA$ to $\mathbf{1}$.
\item[(iii)] $\cloB$ contains a cyclic operation, that is, an operation $t$ of arity $n \geq 2$ such that $t(x_1, \dots x_n) \approx t(x_2, \dots, x_n,x_1)$. 
\item[(iii')] $\cloA$ contains a cyclic operation. 
\item[(iv)] $\cloB$ contains a Siggers operation, that is, a $4$-ary operation $t$ such that $t(a,r,e,a) \approx t(r,a,r,e)$.
\item[(iv')] $\cloA$ contains a Siggers operation.
\end{itemize}
\end{theorem}

In particular, the tractability conjecture can be equivalently formulated as follows. 

\begin{conj} \label{conj:new_finite}
Let $\relA$ be a finite relational structure. Then one of the following holds.
\begin{itemize}
\item The polymorphism clone of $\relA$ maps to $\mathbf{1}$ via an h1 clone homomorphism (and consequently $\CSP(\relA)$ is NP-complete).
\item $\CSP(\relA)$ is solvable in polynomial-time.
\end{itemize}
\end{conj}

\subsection{The infinite case} For countable $\omega$-categorical structures $\relA$ and $\relB$ we have the same general reductions (a), (b), and ({c}) -- only the notion of a \emph{core} has to be replaced by that of a \emph{model-complete core}. A large part of the research on CSPs of infinite structures investigates when a given finite structure $\relB$ can be obtained from an infinite structure $\relA$ via those constructions~\cite{BPP-projective-homomorphisms}. For what concerns the reduction by pp-interpretations, we have the following theorem from~\cite{Topo-Birk}.

\begin{theorem} \label{thm:old_infinite}
Let $\relA$ be a countable $\omega$-categorical and $\relB$ be a finite relational structure, and let $\cloA$, $\cloB$ their polymorphism clones.
Then the following are equivalent.
\begin{itemize}
\item[(i)] $\relB$ is pp-interpretable in $\relA$.
\item[(ii)] $\cloB \in \EHSPfin \cloA $.
\item[(iii)] There exists a continuous clone homomorphism from $\cloA$ into $\cloB$, i.e., a  continuous mapping $\cloA \to \cloB$ preserving identities. 
\end{itemize}
\end{theorem}

Regarding applicability to CSPs, this theorem suffers from the same shortcomings as its finite counterpart, as discussed above. But in the infinite case, two other unsatisfactory features which are not present in the situation for finite structures arise in addition. 

Firstly, the class of all infinite structures being too vast to be approached as a whole, 
research on CSPs of infinite structures focusses on structures with particular properties, such as the Ramsey property or finite boundedness; cf.~for example~\cite{BP-reductsRamsey}. By assuming that a structure $\relA$ is a model-complete core one might lose these properties. In other words, if we start with an $\omega$-categorical structure $\relA$ satisfying a certain property such as the Ramsey property, then the unique model-complete core which is homomorphically equivalent to $\relA$ might fail to satisfy this property.
% Apart from the fact that the existence of the model-complete core of $\relA$ requires $\omega$-categoricity, 
%Passing to this new structure 
This results in a serious technical disadvantage: %after all, the properties required on $\relA$ were required for a reason, and 
much of the machinery developed for the investigation of infinite CSPs cannot be applied, for example, in the absence of the Ramsey property.

Secondly, contrary to the situation in the finite, adding constants to a model-complete core does not terminate after a finite number of steps in the infinite case. Hence, while there is an analog of the concept of a core for the infinite, namely that of a model-complete core, there is no analog of the notion of an idempotent core, or an idempotent polymorphism clone, for the $\omega$-categorical setting. 
 This leads to less elegant formulations than in the finite, such as in the following conjecture of Bodirsky and Pinsker (cf.~\cite{BPP-projective-homomorphisms}).

\begin{conj}\label{conj:old}
Let $\relA$ be a reduct of a finitely bounded homogeneous structure, and let $\relB$ be its model-complete core. Then one of the following holds.
\begin{itemize}
\item There exist elements $b_1,\ldots,b_n$ in $\relB$ such that the polymorphism clone of the expansion of $\relB$ by those constants maps homomorphically and continuously to $\mathbf 1$ (and consequently $\CSP(\relA)$ is NP-complete).
\item $\CSP(\relA)$ is solvable in polynomial-time.
\end{itemize}
\end{conj}

We are going to prove the following theorem which will avoid the issues raised above.

\begin{theorem} \label{thm:new_infinite}
Let $\relA$ be an at most countable $\omega$-categorical and $\relB$ be a finite relational structure, and let $\cloA$, $\cloB$ their polymorphism clones.
Then the following are equivalent.
\begin{itemize}
\item[(i)] $\relB$ is homomorphically equivalent to a pp-power of $\relA$, or equivalently, $\relB$ can be obtained from $\relA$ by a finite number of constructions among (a),~(b),~({c}).
\item[(ii)] $\cloB \in \ETraPfin \cloA$, or equivalently, $\cloB \in \ETraP \cloA$.%; here, $\Tra$ denotes the new operator of taking \transfers{}.
\item[(iii)] There exists a~uniformly continuous h1 clone homomorphism from $\cloA$ into $\cloB$, i.e., a~uniformly continuous mapping $\cloA \to \cloB$ preserving identities of height $1$. 
\end{itemize}
\end{theorem}

This allows us, in particular, to formulate the following conjecture, which is implied by Conjecture~\ref{conj:old} but not necessarily equivalent to it; we refer to Section~\ref{sec:wrapup} for more details.

\begin{conj}\label{conj:new}
Let $\relA$ be a reduct of a finitely bounded homogeneous structure, and let $\cloA$ be its polymorphism clone. Then one of the following holds.
\begin{itemize}
\item $\cloA$ maps to $\mathbf{1}$ via a uniformly continuous h1 clone homomorphism (and consequently $\CSP(\relA)$ is NP-complete).
\item $\CSP(\relA)$ is solvable in polynomial-time.
\end{itemize}
\end{conj}

\subsection{Coloring of clones by relational structures}\label{sect:intro_colorings}

It came as a big surprise to us that there is a~tight connection between  h1 clone homomorphisms and Sequeira's 
notion of compatibility with projections~\cite{sequeira01} which he used to attack some open problems concerning Maltsev conditions.

Each Maltsev condition determines a filter in the lattice of interpretability types of varieties~\cite{neumann74,garcia.taylor84}. The question whether a given Maltsev condition is implied by the conjunction of two strictly weaker conditions translates into the question whether the corresponding filter is join prime (that is, whether the complement of the filter is closed under joins). 
Garcia and Taylor~\cite{garcia.taylor84} conjectured that two important Maltsev conditions determine join prime filters:

\begin{conjs} \label{conj:taylor}
\ 

\begin{itemize}
  \item The filter of congruence permutable varieties is join prime.
  \item The filter of congruence modular varieties is join prime.
\end{itemize}
\end{conjs}

The first conjecture was confirmed by Tschantz in an unpublished paper~\cite{tschantz}. His proof is  technically extremely difficult and it seems impossible to generalize the arguments to even slightly more complex Maltsev conditions, such as the one characterizing $3$-permutability. 
In an effort to resolve the second conjecture and similar problems, Sequeira introduced the notion of compatibility with projections~\cite{sequeira01} and used it to prove some interesting partial results. We generalize his concept to ``colorability of clones by relational structures'' and provide a simple link to h1 clone homomorphisms:

\begin{theorem}\label{thm:coloring-and-h1}
Let $\cloA$ be a function clone and let $\cloB$ be the polymorphism clone of a relational structure $\relB$.  Then the following are equivalent.
\begin{itemize}
\item[(i)] There exists an h1 clone homomorphism from $\cloA$ into $\cloB$.
\item[(ii)] $\cloA$ is $\relB$-colorable.
\end{itemize}
\end{theorem}

As a corollary, we obtain the following generalization of~\cite{bentz.sequeira14} where it was additionally assumed that the varieties are idempotent.

\begin{theorem} \label{thm:perm-and-modular}
Let $\var V$ and $\var W$ be two varieties defined by identities of height at most~1. 
\begin{itemize}
  \item If $\var V$, $\var W$ are not congruence modular, then neither is $\var V \vee \var W$. 
  \item If $\var V$, $\var W$ are not congruence $n$-permutable for any $n$, then neither is $\var V \vee \var W$.
\end{itemize}
\end{theorem}

%Our results shed some light

\subsection{Outline of the article} Section~\ref{sect:prelims} contains the
definitions of the standard notions we used in the introduction. In
Section~\ref{sec:rel} we discuss relational constructions, in particular
pp-powers. Then follow algebraic constructions, notably \transfers{} and the
operator $\Tra$\footnote{Often denoted by $\mathsf D$, for \emph{double shrink}.}, in Section~\ref{sec:tra}. Their syntactic counterpart, h1 clone
homomorphisms, is dealt with in Section~\ref{sec:birk}.  The
$\omega$-categorical case is discussed together with topological considerations
in Section~\ref{sec:cont}. Section~\ref{sec:snek} is devoted to the connection
between colorings and h1 clone homomorphisms.  We conclude this work with
Section~\ref{sec:wrapup} where we prove and discuss the results and conjectures
from the introduction.

\section{Preliminaries}\label{sect:prelims}

We explain the classical notions which appeared in the introduction, and fix some notation for the rest of the article.  The new notion of pp-power, the operator $\Tra$ acting on function clones,  h1 clone homomorphisms, and colorings of clones by relational structures will be defined in their own sections. 
For undefined universal algebraic concepts and more detailed presentations of the notions presented here we  refer to~\cite{BS81,Berg11}. For notions from model theory we refer to~\cite{Hodges}.

\subsection{Relational structures and polymorphism clones} We denote relational structures by $\relA, \relB$, etc. When $\relA$ is a relational structure, we reserve the symbol $A$ for its domain. We write $\cloA$ for its \emph{polymorphism clone}, i.e., the set of all finitary operations on $A$ which preserve all relations of $\relA$, usually denoted by $\Pol(\relA)$ in the literature. The polymorphism clone $\cloA$ is always a \emph{function clone}, i.e., it is closed under composition and contains all projections.

\subsection{CSPs} For a finite relational signature $\Sigma$ and a $\Sigma$-structure $\relstr{A}$, the \emph{constraint satisfaction problem} of $\relstr{A}$, or $\CSP(\relstr{A})$ for short, is the membership problem for the class
\[
\{ \relstr{C}: \relstr{C} \mbox{ is a finite $\Sigma$-structure and there exists a~homomorphism $\relstr{C} \to \relstr{A}$}\} \enspace.
\]
An alternative definition of $\CSP(\relstr{A})$ is via primitive positive (pp-) sentences. Recall that a \emph{pp-formula} over $\relstr{A}$ is a first order formula which only uses predicates from $\relstr{A}$, conjunction, equality, and existential quantification. $\CSP(\relstr{A})$ can equivalently be phrased as the membership problem of the set of  pp-sentences which are true in $\relstr{A}$. 

Our results, in particular Theorems~\ref{thm:new_finite} and~\ref{thm:new_infinite}, are purely structural and complexity-free. Still it is worthwhile mentioning that when we say that a computational problem reduces to, or is not harder than, another computational problem, we have log-space reductions in mind (although the claim is true also for other meanings of hardness).

\subsection{Notions from the infinite}\label{sect:notions_infinite} The algebraic theory of the CSP is best developed for finite structures and countable $\omega$-categorical structures. 
Recall that an at most countable relational structure $\relstr{A}$ is \emph{$\omega$-categorical} if, for every $n \geq 1$, the natural componentwise action of its automorphism group $\Aut(\relstr{A})$ on $A^n$ has only finitely many orbits. In particular finite structures are always $\omega$-categorical.

The class of infinite $\omega$-categorical structures which has probably received most attention in the literature are reducts of finitely bounded homogeneous structures, which appear in Conjectures~\ref{conj:old} and~\ref{conj:new}. Since we do not need these notions in the present paper, we refrain from defining them and refer to the surveys~\cite{BP-reductsRamsey},~\cite{Bodirsky-HDR}, and~\cite{Pin15}.

\subsection{pp definitions and interpretations} A relation is \emph{pp-definable} in a  relational structure $\relstr{A}$ if it is definable with a pp-formula over $\relstr{A}$ without parameters. Let $\relstr{A}$, $\relstr{B}$ be relational structures with possibly different signatures. We say that $\relstr{A}$ \emph{pp-interprets} $\relstr{B}$, or that $\relstr{B}$ is \emph{pp-interpretable} in $\relstr{A}$,  if there exists $n \geq 1$ and a mapping $f$ from a subset of $A^n$ onto $B$ such that the following relations are pp-definable in $\relstr{A}$:
\begin{itemize}
\item the domain of $f$;
\item the preimage of the equality relation on $B$ under $f$, viewed as a $2n$-ary relation on $A$;
\item the preimage of every relation in $\relstr{B}$ under $f$, where the preimage of a $k$-ary relation under $f$ is again regarded as a $kn$-ary relation on $A$.
\end{itemize}

\subsection{Homomorphic equivalence and cores}

When relational structures $\relstr{A}$ and $\relstr{B}$ have the same signature, then we say that $\relstr{A}$ and $\relstr{B}$ are \emph{homomorphically equivalent} if there exists a~homomorphism $\relstr{A} \to \relstr{B}$ and a~homomorphism $\relstr{B} \to \relstr{A}$. A relational structure $\relstr{B}$ is called a \emph{model-complete core} if the automorphisms of~$\relstr{B}$ are dense in its endomorphisms, i.e., for every endomorphism $e$ of $\relstr{B}$ and every finite subset $B'$ of $B$ there exists an automorphism of~$\relstr{B}$ which agrees with $e$ on $B'$. When $\relstr{B}$ is finite, then this means that every endomorphism is an automorphism, and $\relstr{B}$ is simply called a \emph{core}. Every at most countable $\omega$-categorical structure is homomorphically equivalent to a unique model-complete core, which is again $\omega$-categorical~\cite{Cores-journal, JBK}. A special case of a core is an  \emph{idempotent core}, by which we mean a relational structure whose only endomorphism is the identity function on its domain. By adding all unary singleton relations to a finite core (recall reduction~({c}) from the introduction) one obtains an idempotent core. Note that on the other hand, a countably infinite $\omega$-categorical structure is never idempotent since its automorphism group is \emph{oligomorphic}, i.e., large in a certain sense; recall Section~\ref{sect:notions_infinite}.

\subsection{HSP and E}
When $\cloA$ is a function clone, then we denote by $\HHH(\cloA)$ all function clones obtained by letting $\cloA$ act naturally on the classes of an invariant equivalence relation on its domain. By $\SSS(\cloA)$ we denote all function clones obtained by letting $\cloA$ act on an an invariant subset of its domain via restriction. We write $\PPP(\cloA)$ and $\PPPfin(\cloA)$ for all componentwise actions of $\cloA$ on powers and finite powers of its domain, respectively. The operator $\EEE(\cloA)$ yields all function clones obtained from $\cloA$ by adding functions to it. All these operators are to be understood up to renaming of elements of domains, i.e., we consider two function clones equal if one can obtained from the other via a bijection between their respective domains. We use combinations of these operators, such as $\HSP(\cloA)$, with their obvious meaning.

We denote algebras by $\alg{A}$, $\alg{B}$, etc., and their domains by $A$, $B$, etc. We apply the operators $\HHH$, $\SSS$, $\PPP$, and $\PPPfin$  also to algebras and to classes of algebras of the same signature as it is standard in the literature. When we apply $\PPP$ to a class of algebras, then we refer to all products of algebras in the class (rather than powers only). 
Concerning the operator $\EEE$, we shall also apply it to algebras and sets of
functions which are not necessarily function clones, and mean that it returns
all algebras whose operations contain all operations the original algebra (all
sets of functions that contain the original set of functions).

\subsection{Clone homomorphisms} A \emph{clone homomorphism} from a function clone $\cloA$ to a function clone $\cloB$ is a mapping $\xi\colon \cloA\To\cloB$ which
\begin{itemize}
\item preserves arities, i.e., it sends every function in $\cloA$ to a function of the same arity in~$\cloB$;
\item preserves each projection, i.e., it sends the $k$-ary projection onto the $i$-th coordinate in $\cloA$ to the same projection in $\cloB$, for all $1\leq i\leq k$;
\item preserves composition\footnote{Also see the comment below Definition
\ref{def:5.2}.}, i.e., $\xi(f(g_1,\ldots,g_n))=\xi(f)(\xi(g_1),\ldots,\xi(g_n))$ for all $n$-ary functions $f$ and all $m$-ary functions $g_1,\ldots,g_n$ in $\cloA$.
\end{itemize}
For all $1\leq i\leq k$ we denote the $k$-ary projection onto the $i$-th coordinate by $\pi^k_i$, in any function clone and irrespectively of the domain of that clone. This slight abuse of notation allows us, for example, to express the second item above by writing $\xi(\pi^k_i)=\pi^k_i$.

\subsection{The lattice of interpretability types of varieties}
A \emph{variety} is a class of algebras of the same signature closed under homomorphic images, subalgebras, and products. 
By Birkhoff's HSP theorem~\cite{Bir-On-the-structure}, a class of algebras of the same signature is a variety if and only if it is the class of models of some set of identities (where an \emph{identity} is a universally quantified equation). Each variety $\var V$ has a \emph{generator}, that is, an algebra $\alg A \in \var V$ such that $\var V = \HSP \algA$. We denote by $\varclo V$ the clone of term operations of any generator of $\var V$ (the choice of generator is immaterial for our purposes). 

We quasi-order the class of all varieties (of varying signatures) by defining $\var V \leq \var W$ if there exists a clone homomorphism $\varclo{\var V} \to \varclo{\var W}$. Then we identify two varieties $\var V$, $\var W$ if $\var V \leq \var W \leq \var V$. The obtained partially ordered class is a lattice -- the \emph{lattice of interpretability types of varieties}. The lattice join can be described by means of the defining identities: the signature of $\var V \vee \var W$ is the disjoint union of the signatures of $\var V$ and $\var W$ and the set of defining identities of $\var V \vee \var W$ is the union of the defining identities of $\var V$ and $\var W$. 

\subsection{Topology} Every function clone is naturally equipped with the topology of pointwise convergence: in this topology, a sequence $(f_i)_{i\in\omega}$ of $n$-ary functions converges to an $n$-ary function $f$ on the same domain if and only if for all $n$-tuples $\bar a$ of the domain the functions $f_i$ agree with $f$ on $\bar a$ for all but finitely many $i\in\omega$. Therefore, every function clone gives rise to an abstract topological clone~\cite{Reconstruction}. %On an at most countable domain, this topology on function clones is Polish, i.e., separable and completely metrizable. 

We imagine function clones always as carrying this topology, which is, in the case of a~countable domain, in fact induced by a~metric, and in general by a~uniformity~\cite{Reconstruction}. Then a~mapping $\xi\colon \cloA\To\cloB$, where $\cloA$ and $\cloB$ are function clones, is continuous if and only if for all $f\in\cloA$ and all finite sets $B'\subseteq B$ there exists a~finite set $A'\subseteq A$ such that for all $g\in\cloA$ of the same arity as $f$, if $g$ agrees with $f$ on $A'$, then $\xi(g)$ agrees with $\xi(f)$ on $B'$. It is uniformly continuous if and only if for all finite sets $B'\subseteq B$ there exists a~finite set $A'\subseteq A$ such that whenever two functions $f, g\in\cloA$ of the same arity agree on $A'$, then their images $\xi(f), \xi(g)$ under $\xi$ agree on~$B'$.

We remark that the polymorphism clones of relational structures are precisely the  function clones which are complete with respect to this topology. Function clones on a finite domain are discrete.

\section{Relational Constructions} \label{sec:rel}

We first recall the general CSP reductions mentioned in the introduction, there labelled (a), (b), and~({c}), and then introduce a~weaker variant of a pp-interpretation which we call \emph{pp-power}.

\subsection{Classical reductions} The first reduction~(a) from the introduction is the one via primitive positive interpretations, justified by the following proposition~\cite{JBK}.

\begin{proposition} Let $\relstr{A}$, $\relstr{B}$ are relational structures with finite signatures. If $\relstr{A}$ pp-interprets $\relstr{B}$, then $\CSP(\relstr{B})$ is log-space reducible to $\CSP(\relstr{A})$. 
\end{proposition}

The second reduction,~(b) in the introduction, is homomorphic equivalence. We have the following observation, which is easily verified using the fact that  homomorphisms preserve pp-formulas.

\begin{proposition}
Let $\relstr{A}$, $\relstr{B}$ be relational structures in the same finite signature which are homomorphically equivalent. Then $\CSP(\relstr{A})=\CSP(\relstr{B})$.
\end{proposition}

Let us observe here that  for the CSPs of $\relstr{A}$ and $\relstr{B}$ to be equal, we only have to require that every finite substructure of $\relstr{A}$ maps homomorphically into $\relstr{B}$ and vice-versa. This gives us a more general reduction in general; however, in the  countable $\omega$-categorical case this condition of ``local'' homomorphic equivalence is easily seen to be equivalent to full homomorphic equivalence via a standard compactness argument, and so we shall not further consider this notion in the present paper.

%Every at most countable $\omega$-categorical structure is homomorphically equivalent to an up to isomorphism unique \emph{model-complete core}, that is a structure $\relstr{B}$ whose set of endomorphisms is dense in the set of automorphisms~\cite{Cores-journal, JBK}. 

The following proposition from~\cite{JBK} states that an $\omega$-categorical model-complete core can be expanded by a singleton relation without making the CSP harder, giving us reduction~({c}) from the introduction. Although stated in~\cite{JBK} for finite structures, it holds also in the $\omega$-categorical case (cf.~\cite{Bodirsky-HDR}).

\begin{proposition}
Let $\relstr{A}$ be an at most countable $\omega$-categorical structure which is a model-complete core and let $\relstr{B}$ be a structure obtained from $\relstr{A}$ by adding a singleton unary relation. The $\CSP(\relstr{A})$ and $\CSP(\relstr{B})$ are log-space equivalent. 
\end{proposition}

The following definition assembles all three reductions.

\begin{definition}\label{defn:ppconstructible}
Let $\relstr{A}$, $\relstr{B}$ be relational structures. We say that $\relstr{B}$ can be \emph{pp-constructed} from $\relstr{A}$
if there exists a sequence $\relstr{A} = \relstr{C}_1, \relstr{C}_2, \dots, \relstr{C}_k = \relstr{B}$ such that for every $1 \leq i < k$
\begin{itemize}
  \item $\relstr{C}_i$ pp-interprets $\relstr{C}_{i+1}$, or
	\item $\relstr{C}_{i+1}$ is homomorphically equivalent to $\relstr{C}_i$, or
	\item $\relstr{C}_i$ is an at most countable $\omega$-categorical model-complete core, and $\relstr{C}_{i+1}$ is obtained from $\relstr{C}_i$ by adding a singleton unary relation.
\end{itemize}
\end{definition}

\begin{corollary}\label{cor:complexity}
Let $\relstr{A}$, $\relstr{B}$ be relational structures. If $\relstr{B}$ can be pp-constructed from $\relstr{A}$, then $\CSP(\relstr{B})$ is log-space reducible to $\CSP(\relstr{A})$.
\end{corollary}

\subsection{pp-powers}

We now introduce a weakening of pp-interpretations which together with homomorphic equivalence will cover all classical reductions.
This weakening is obtained by requiring that the partial surjective mapping $f$ from the definition of pp-interpretation in Section~\ref{sect:prelims} is a bijection with full domain. 

\begin{definition}
Let $\relstr{A}, \relstr{B}$ be relational structures. 
%We say that $\relstr{B}$ is \emph{pp-definable} over $\relstr{A}$ if $A=B$ and all relations of $\relstr{B}$ are definable in $\relstr{A}$ with a primitive positive formula without parameters. 
We say that $\relstr{B}$ is a \emph{pp-power} of $\relstr{A}$ if it is isomorphic to a structure with domain $A^n$, where $n\geq 1$, whose relations are pp-definable from $\relstr{A}$; as before, a $k$-ary relation on $A^n$ is regarded as a $kn$-ary relation on $A$. 
\end{definition}

The next lemma deals with reductions (a) and (b) and their combinations. We need an auxiliary definition.

\begin{definition}
For a class $\class{K}$ of relational structures, we denote by
\begin{itemize}
\item $\PpInt \class{K}$ the class of structures which are pp-interpretable in some structure in $\class{K}$;
\item $\PpPower \class{K}$ the class of pp-powers of structures in $\class{K}$;
\item $\HomEq \class{K}$ the class of structures which are homomorphically equivalent to a member of $\class{K}$.
\end{itemize}
\end{definition}

\begin{lemma} \label{lem:pppp}
Let $\class{K}$ be a class of relational structures. Then
\begin{enumerate}
\item[(i)] $\PpInt \class{K} \subseteq \HomEq \PpPower \class{K}$;
\item[(ii)] $\HomEq \HomEq \class{K} = \HomEq \class{K}$;
\item[(iii)] $\PpPower \PpPower \class{K} = \PpPower \class{K}$;
\item[(iv)] $\PpPower \HomEq \class{K} \subseteq \HomEq \PpPower \class{K}$.
\end{enumerate}
\end{lemma}
\begin{proof}
To show (i), let $\relstr{A}$, $\relstr{B}$ be relational structures such that $\relstr{A}$ pp-interprets $\relstr{B}$ and let $f\colon A^n \to B$ be a partial surjective mapping witnessing the pp-interpretability. Take the relational structure $\relstr{C}$ with the same signature as $\relstr{B}$ whose universe is $C = A^n$ and whose relations are $f$-preimages of relations in $\relstr{B}$. Let $f'$ be any mapping from $A^n$ to $B$ extending $f$ and let $g$ be any mapping from $B$ to $A^n$ such that $f'\circ g$ is the identity on $B$. Now $\relstr{C}$ is a pp-power of $\relstr{A}$, $f'$ is a homomorphism from $\relstr{C}$ to $\relstr{B}$ and $g$ is a homomorphism from $\relstr{B}$ to $\relstr{C}$. Thus $\relstr{B} \in \HomEq \PpPower \relstr{A}$, as required.

Item (ii) follows from the fact that the composition of two homomorphisms $\relstr{A} \to \relstr{B} \to \relstr{C}$ is a homomorphism $\relstr{A} \to \relstr{C}$. 
Item (iii) is readily seen as well.

Finally, let $\relstr{A}$ and $\relstr{B}$ be homomorphically equivalent relational structures of the same signature and let $\relstr{C}$ be an $n$-th pp-power of $\relstr{B}$. We define an $n$-th pp-power $\relstr{D}$ of $\relstr{A}$ of the same signature as $\relstr{C}$ as follows: for each relation in $\relstr{C}$, consider some pp-definition from $\relstr{B}$ and replace all relations in this pp-formula by the corresponding relations of $\relstr{A}$. It is easy to see that the $n$-th power of any homomorphism from $\relstr{A}$ to $\relstr{B}$ (from $\relstr{B}$ to $\relstr{A}$, respectively) is a homomorphism from $\relstr{D}$ to $\relstr{C}$ (from $\relstr{C}$ to $\relstr{D}$, respectively). In particular, $\relstr{C}$ and $\relstr{D}$ are homomorphically equivalent, proving (iv).
\end{proof}

The following lemma shows that reduction~({c}), namely the adding of singleton unary relations to model-complete cores, is actually already covered by homomorphic equivalence and pp-interpretations; in particular, we could have omitted it in Definition~\ref{defn:ppconstructible}.

\begin{lemma} \label{lem:adding_const}
Let $\relstr{A}$ be an at most countable $\omega$-categorical structure which is a~model-complete core and let $\relstr{B}$ be a structure obtained from $\relstr{A}$ by adding a singleton unary relation. Then $\relstr{B} \in \HomEq \PpPower \relstr{A}$.
\end{lemma}

\begin{proof}
We denote $S = \{s\}$, $s \in A$, the added singleton unary relation. Let $O \subseteq A$ denote the orbit of $s$ under the action of $\Aut(\relstr{A})$. Since $\relstr{A}$ is a model-complete core, $O$ is preserved by all endomorphisms of $\relstr{A}$, and since $O$ consists of only one orbit with respect to the action of $\Aut(\relstr{A})$, this implies that $O$ is preserved by all polymorphisms of $\relstr{A}$~\cite{tcsps-journal}. Hence, using $\omega$-categoricity we infer that  $O$ is pp-definable in $\relstr{A}$~\cite{BodirskyNesetrilJLC}.

We define a relational structure $\relstr{C}$ over the domain $C = A^2$ with the same signature as $\relstr{B}$. The relation of $\relstr{C}$ corresponding to a $k$-ary relation $R$ in $\relstr{A}$ is defined by
\[
\overline{R} = \{((a_1,b_1), \dots, (a_k,b_k)) \in (A^2)^k : b_1=b_2=\dots=b_k \in O, (a_1,\dots, a_k) \in R\}
\]
and the relation corresponding to the singleton unary relation $S$ is defined as $\overline{S} = \{(a,a): a \in O\}$. 
Clearly, $\relstr{C}$ is a pp-power of $\relstr{A}$, so it remains to show that $\relstr{C}$ is homomorphically equivalent to $\relstr{B}$.

The mapping $A \to A^2$, defined by $a \mapsto (a,s)$ is a homomorphism from $\relstr{B}$ to $\relstr{C}$.

To define a homomorphism $f\colon \relstr{C} \to \relstr{B}$ we first pick, for each $a \in O$, an automorphism $\alpha_a \in \Aut(\relstr{A})$ with $\alpha_a(a) = s$. Now $f$ is defined by $f(a,b) = \alpha_b(a)$ if $b \in O$, and otherwise arbitrarily. We check that this mapping is indeed a homomorphism. If $R$ is a $k$-ary relation in $\relstr{A}$ and $\mathbf{x} = ((a_1,b_1), \dots, (a_k,b_k)) \in \overline{R}$, then 
$b_1=b_2= \dots=b_k=b \in O$ and we have $f(\mathbf{x}) = (\alpha_b(a_1), \dots, \alpha_b(a_k))$. This $k$-tuple is in $R$ as $(a_1, \dots, a_k) \in R$ and $\alpha_b$ is an automorphism of $\relstr{A}$. Finally, $f$ also preserves the added relation since each $(a,a) \in \overline{S}$ is mapped to $\alpha_a(a) = s \in S$.
\end{proof}

A corollary of the last two lemmata is that we can cover all general reductions to a given structure by considering all structures which are homomorphically equivalent to a pp-power of that structure.

\begin{corollary}\label{cor:relational}
The following are equivalent for at most countable $\omega$-categorical relational structures $\relstr{A}, \relstr{B}$.
\begin{itemize}
\item[(i)] $\relstr{B}$ can be pp-constructed from $\relstr{A}$.
\item[(ii)] $\relstr{B} \in \HomEq \PpPower \relstr{A}$; that is, $\relstr{B}$ is homomorphically equivalent to a pp-power of $\relstr{A}$.
\end{itemize}
\end{corollary}

\begin{proof}
Lemma~\ref{lem:pppp}, item (i), and Lemma~\ref{lem:adding_const} imply that if 
 $\relstr{B}$ can be pp-constructed from $\relstr{A}$, then it can be constructed using pp-powers and homomorphic equivalence. From items (ii), (iii) and (iv) of Lemma~\ref{lem:pppp} we can conclude that $\relstr{B} \in \HomEq \PpPower \relstr{A}$.

The other implication is trivial.
\end{proof}

We shall conclude this section with an example of two finite idempotent relational structures $\relstr{A}$, $\relstr{B}$ such that $\relstr{B} \in \HomEq \PpPower \relstr{A}$ but $\relstr{B} \not\in \PpInt \relstr{A}$.  Therefore, as mentioned in the introduction, the common practice of first reducing to cores and then considering only pp-interpretations results in a true loss of power.

\begin{example}\label{ex:hepp}
Consider the relational structure $\relstr{A}$ with domain $A = \mathbb{Z}_2^2$ consisting of ternary relations $R_{(a,b)}$, $(a,b) \in \mathbb{Z}_2^2$ defined by
\[
R_{(a,b)} = \{(\mathbf{x},\mathbf{y},\mathbf{z}) \in (\mathbb{Z}_2^2)^3: \mathbf{x} + \mathbf{y} + \mathbf{z} = (a,b)\},
\]
and unary singleton relations $\{(a,b)\}$, $(a,b) \in \mathbb{Z}_2^2$. Let $\relstr{A}'$ be the reduct of $\relstr{A}$ formed by the relations $R_{(0,0)}$, $R_{(1,0)}$, $\{(0,0)\}$, and $\{(1,0)\}$. Trivially, $\relstr{A}'$ is pp-definable from $\relstr{A}$.
Finally, take the relational structure $\relstr{B}$ with $B = \mathbb{Z}_2$ and relations $R_a$, $a \in \mathbb{Z}_2$, where
\[
R_a = \{(x,y,z) \in \mathbb{Z}_2^3: x+y+z = a\},
\]
together with singletons $\{0\}$, $\{1\}$.

Note that $\clone{A}$ consists of idempotent affine operations of the module $\mathbb{Z}_2^2$ over $\End(\mathbb{Z}_2^2)$
and $\clone{B}$ is formed by idempotent affine operations of the abelian group $\mathbb{Z}_2$.

The mappings $A \to B$, $(x_1,x_2) \mapsto x_1$ and $B \to A$, $x \mapsto (x,0)$ are homomorphisms from $\relstr{A}'$ to $\relstr{B}$ and from $\relstr{B}$ to $\relstr{A}'$, respectively. (In fact, $\relstr{B}$ is the core of $\relstr{A}'$.) Therefore, $\relstr{B} \in \HomEq \PpPower \relstr{A}$.

Suppose  that $\relstr{A}$ pp-interprets $\relstr{B}$ and $f$ is a mapping from $C \subseteq A^n$ onto $B$ witnessing this. Let $\alpha$ be the kernel of $f$ -- it is an equivalence relation on $C$ with two blocks.
By the definition of pp-interpretation, both $C$ and $\alpha$ are pp-definable from $\relstr{A}$. Since $\relstr{A}$ contains the singleton unary relations, then both blocks of $\alpha$ are pp-definable as well. Thus $C$ is a pp-definable relation which is a disjoint union of two pp-definable relations. This is impossible as it is easily seen that the cardinality of any relation pp-definable from $\relstr{A}$ is a power of $4$.
\end{example}

\section{Algebraic Constructions} \label{sec:tra}

In the light of Theorems~\ref{thm:old_finite} and~\ref{thm:old_infinite}, one can say that the algebraic counterpart of pp-interpretations in $\omega$-categorical structures is roughly the $\HSPfin$ operator. We introduce a new algebraic operator which corresponds to homomorphic equivalence.

\begin{defn}
Let $\alg{A}$ be an algebra with signature $\tau$. Let  $B$ be a set, and let 
 $h_1\colon B\To A$ and $h_2 \colon A\To B$ be functions. We define a $\tau$-algebra $\alg{B}$ with domain $B$ as follows: for every operation $f(x_1,\ldots,x_n)$ of $\alg{A}$,  $\alg{B}$ has the operation
$$
(x_1,\ldots,x_n)\mapsto h_2(f(h_1(x_1),\ldots,h_1(x_n))).
$$
We call $\alg{B}$ a~\emph{\transfer{}}
 of $\alg{A}$. If $h_2\circ h_1$ is the identity function on $B$, then 
we say that $\alg{B}$ is a~\emph{retraction}
 of $\alg{A}$. 
\end{defn}

\begin{defn}
For a class $\class{K}$ of algebras, 
we write $\Tra \class{K}$ for all \transfers{} of algebras in $\class{K}$, and similarly $\Ret \class{K}$ for all retractions of algebras in $\class{K}$. 
 We also apply this operator to single algebras, writing $\Tra \alg{A}$, with the obvious meaning. 
\end{defn}

By viewing the operations of a function clone as those of an algebra, we moreover apply the operator to function clones, similarly to the operators $\HHH$, $\SSS$, and $\PPP$. Observe, however, that contrary to the situation with the latter operators, 
the \transfer{} of a~function clone need not be a function clone since it need not contain the projections or be closed under composition. 

\begin{defn}
For a function clone $\cloA$, we denote by $\Tra\cloA$ all sets of functions obtained as in the above definition.
\end{defn}

%Clearly, a shrink is a special case of a double shrink. When $h_2$ is injective, then a double shrink is just a shrink. When $h_2$ is not injective, then its kernel is a congruence relation of $\B$, and so $\B$ could as well be obtained by first applying a shrink and then blowing up the elements of the algebra obtained to classes. Finally, observe that the double shrink can increase the size of the domain, whereas the shrink cannot. In particular, one can obtain an infinite algebra by double shrinking a finite algebra, which is impossible by the shrink.

It is well-known that for any class of algebras $\class{K}$,  the class $\HSP \class{K}$ is equal to the closure of $\class{K}$ under $\HHH$, $\SSS$ and $\PPP$~\cite{Bir-On-the-structure}. 
We now provide a similar observation which includes the operator $\Tra$. Note that in particular, the first item of the following lemma implies that the operator $\Tra$ is a common generalization of the operators $\HHH$ and $\SSS$.

\begin{lemma}\label{lem:algebraicinclusions}
Let $\class{K}$ be a class of algebras of the same signature. 
\begin{itemize}
\item[(i)] $\HHH \class{K} \subseteq \Tra \class{K}$\, and\,  $\SSS \class{K} \subseteq \Tra \class{K}$;
\item[(ii)] $\RR \class{K} \subseteq \Tra \class{K}$;
\item[(iii)] $\PR \class{K} \subseteq \RP \class{K}$\, and\, $\Pfin \Tra \class{K} \subseteq \RPfin \class{K}$.
\end{itemize}
Analogous statements hold for the operator $\Ret$ instead of $\Tra$.
\end{lemma}
\begin{proof}
If $\alg{B}\in \HHH \class{K}$, then there exists an algebra $\alg{A}\in \class{K}$ and a surjective homomorphism $h\colon A\To B$. Picking any function $h'\colon B\To A$ such that $h\circ h'$ is the identity function on $B$ and setting $h_1:=h'$ and $h_2:=h$ we then see that $\alg{B}$ is a \transfer{}, and indeed even a retraction, of $\alg{A}$.

Now let $\alg{B}\in \SSS \class{K}$, and let $\alg{A}\in \class{K}$ such that $B\subseteq A$ is an invariant under the operations of $\alg{A}$ and such that $\alg{B}$ arises by restricting the operations of $\alg{A}$ to $B$. Setting $h_1$ to be the identity function on $B$, and $h_2$ to be any extension of $h_1$ to $A$, then shows that $\alg{B}$ is a retraction of $\alg{A}$.

We now show~(ii). Suppose that $\alg{C}$ is a \transfer{} of $\alg{B}$, witnessed by functions $h_1\colon C\To B$ and $h_2\colon B\To C$, and that $\alg{B}$ is a \transfer{} of $\alg{A}$, witnessed by functions $h_1'\colon B\To A$ and $h_2'\colon A\To B$. Then the functions $h_1'\circ h_1\colon C\To A$ and $h_2\circ h_2'\colon A\To C$ witness that $\alg{C}$ is a \transfer{} of $\alg{A}$. The same proof works for retractions instead of \transfers{}.

We turn to the proof of~(iii). Let $I$ be an~arbitrary set, and suppose that $\alg B_i$ is a \transfer{} of $\alg A_i$ for every $i\in I$, witnessed by functions $g_i\colon B_i\To A_i$ and $h_i\colon A_i\To B_i$. We show that the product $\prod_{i\in I}\alg B_i$ is a~\transfer{} of the product $\prod_{i\in I}\alg A_i$. But this is easy: $g_i$ and $h_i$ induce functions $g\colon \prod_{i\in I}B_i\To \prod_{i\in I} A_i$ and $h\colon \prod_{i\in I}A_i\To \prod_{i\in I} B_i$ by letting them act on the respective components. It is easily verified that $\prod_{i\in I}\alg B_i$ is a \transfer{} of $\prod_{i\in I}\alg A_i$ via those functions. This proves also the finite product version of the statement.
%We turn to the proof of~(iii). Suppose that $\alg{B}$ is a \transfer{} of $\alg{A}$, witnessed by functions $h_1\colon B\To A$ and $h_2\colon A\To B$, and let $I$ be any set. We have to show that the power $\alg{B}^I$ is a \transfer{} of a power of $\alg{A}$; in fact, we will show that it is a \transfer{} of the power $\alg{A}^I$. But this is easy: $h_1$ and $h_2$ induce functions $h_1'\colon B^I\To A^I$ and $h_1'\colon A^I\To B^I$ by letting them act on components. It is easily verified that $\alg{B}^I$ is a \transfer{} of $\alg{A}^I$ via those functions. This proves also the finite power version of the statement.
\end{proof}

\begin{corollary}
Let $\class{K}$ be a class of algebras of the same signature. Then
\begin{itemize}
\item  $\RP \class{K}$ is equal to the closure of $\class{K}$ under $\Tra$, $\HHH$, $\SSS$, and $\PPP$;
\item  $\RPfin \class{K}$ is equal to the closure of $\class{K}$ under $\Tra$, $\HHH$, $\SSS$, and $\Pfin$;
\item $\Ret \PPP \class{K}$ is equal to the closure of $\class{K}$ under $\Ret$, $\HHH$, $\SSS$, and $\PPP$;
\item $\Ret \Pfin \class{K}$ is equal to the closure of $\class{K}$ under $\Ret$, $\HHH$, $\SSS$, and $\Pfin$.
\end{itemize}
\end{corollary}
\begin{proof}
This is a direct consequence of Lemma~\ref{lem:algebraicinclusions}.
\end{proof}

The following proposition establishes relationships between the relational and algebraic constructions discussed to far. 
Items (i) and (iii) belong to the fundamentals of the algebraic approach to the CSP, and have been observed in~\cite{JBK} and~\cite{BodirskySurvey}; see also~\cite{Topo-Birk}. The other two items are similar statements for our notion of pp-power  and the operator $\Tra$. 

\begin{proposition}\label{prop:algebraic}
Let $\relstr{A}$, $\relstr{B}$ be at most countable $\omega$-categorical structures. Then
\begin{itemize}
\item[(i)] $\relstr{B}$ is pp-definable from $\relstr{A}$ iff $\clone{B} \in \Exp \clone{A}$;
\item[(ii)] $\relstr{B}$ is a pp-power of $\relstr{A}$ iff $\clone{B} \in \EPfin \clone{A}$;
\item[(iii)] $\relstr{B}$ is pp-interpretable in $\relstr{A}$ iff $\clone{B} \in \EHSPfin \clone{A}$;
\item[(iv)] $\relstr{B}$ is homomorphically equivalent to a structure which is pp-definable from $\relstr{A}$ iff $\clone{B} \in \ETra \clone{A}$.
\end{itemize}
\end{proposition}

\begin{proof}
For (i) and (iii) we refer to the literature, namely~\cite{JBK} and~\cite{BodirskySurvey}. Item~(ii) is a straightforward consequence of~(i).

To prove~(iv), let $\relstr{A}'$ be pp-definable in $\relstr{A}$, and let $h_1\colon B\To A$ and $h_2\colon A\To B$ be homomorphisms from $\relstr{B}$ into $\relstr{A}'$ and vice-versa. We want to show $\clone{B} \in \ETra \clone{A}$. Let $\cloC'$ be the set of all functions on $B$ which are obtained by applying a \transfer{} to $\cloA'$ via the mappings $h_1, h_2$. Because the latter mappings are homomorphisms, it follows that all functions in $\cloC'$ preserve all relations of $\relstr{B}$, and so $\cloC'$ is contained in $\cloB$. By~(i) we have $\cloA'\supseteq \cloA$, and so $\cloC'$ contains the set $\cloC$ of all functions on $B$ which are obtained by applying a \transfer{} to $\cloA$ via the mappings $h_1, h_2$. Hence, $\clone{B} \in \ETra \clone{A}$.

For the other direction, let $h_1\colon B\To A$ and $h_2\colon A\To B$ be so that the \transfer{} $\cloC$ of $\cloA$ by those functions is contained in $\cloB$.
For every relation $R$ of $\relstr{B}$ set
$$
R':=\{f(h_1(r_1),\ldots,h_1(r_n)) : f\in\cloA,\; r_1,\ldots,r_n\in R\};
$$
here, we apply $h_1$ and functions from $\cloA$ to tuples componentwise. In other words, $R'$ is the closure of $h_1[R]$ under $\cloA$. Let $\relstr{A}'$ be the structure on $A$ whose relations are precisely those of this form. By definition, all relations of $\relstr{A}'$ are invariant under functions in $\cloA$, so $\cloA'\in \Exp \clone{A}$ and hence $\relstr{A}'$ is pp-definable in $\relstr{A}$ by~(i). 
Clearly, $h_1$ is a homomorphism from $\relstr{B}$ to $\relstr{A}'$, since $\cloA$ contains the projections. If on the other hand $f(h_1(r_1),\ldots,h_1(r_n))$ is any tuple in a relation $R'$ of $\relstr{A}'$, then  $h_2(f(h_1(r_1),\ldots,h_1(r_n)))\in R$ because $h_2(f(h_1(x_1),\ldots,h_1(x_n)))\in\cloC$ is contained in $\cloB$ and because $R$ is invariant under the functions in $\cloB$.
\end{proof}

The following corollary incorporates all we have obtained so far, characterizing the notion of pp-constructability via algebraic operators.

\begin{corollary} \label{cor:pp_mut}
Let $\relstr{A}$, $\relstr{B}$ be at most countable $\omega$-categorical structures. Then
 $\relstr{B}$ can be pp-constructed from $\relstr{A}$ iff $\clone{B} \in \ETraPfin \clone{A}$. In this case $\CSP(\relstr{B})$ is log-space reducible to $\CSP(\relstr{A})$.
\end{corollary}
\begin{proof}
By Corollary~\ref{cor:relational}, $\relstr{B}$ can be pp-constructed from $\relstr{A}$ iff $\relstr{B}$ is homomorphically equivalent to a pp-power of $\relstr{A}$. By Proposition~\ref{prop:algebraic}, this is the case iff $\cloB$ is contained in $\EREPfin \clone{A}$. Clearly, the latter class equals $\ETraPfin \clone{A}$. The second statement follows by application of Corollary~\ref{cor:complexity}.
\end{proof}

%\begin{defn}
%When $\mathscr C$ is a function clone, then we write $\R(\mathscr C)$, $\D(\mathscr C)$, $\HHH(\mathscr C)$, $\SSS(\mathscr C)$, $\PPP(\mathscr C)$, and $\Pfin(\mathscr C)$ for sets of functions obtained by giving $\mathscr C$ a signature and then applying the respective operators. Except for the first two operators, these sets of functions will again be function clones.
%\end{defn}

\section{Linear Birkhoff} \label{sec:birk}

We now turn to a syntactic characterization of the operator $\Tra$ together with $\PPP$, similar to the syntactic description of the operators $\HHH$, $\SSS$, and $\PPP$ in item~(iii) of Theorem~\ref{thm:old_finite}. Let us recall the notion of a clone homomorphism and introduce two weakenings thereof.

\begin{definition}\label{defn:clonehomo}
Let $\clone{A}$ and $\clone{B}$ be function clones and let $\xi\colon \clone{A} \to \clone{B}$ be a mapping that preserves arities. We say that $\xi$  is 
\begin{itemize}
\item a \emph{clone homomorphism}, or \emph{preserves identities}, if 
\[
\xi(\pi^n_k)=\pi^n_k\quad \text{ and }\quad 
\xi(f(g_1, \dots, g_n)) = \xi(f)(\xi(g_1), \dots, \xi(g_n))
\]
for any $1\leq k \leq n$, any $m\geq 1$, any $n$-ary operation $f \in \clone{A}$, and all $m$-ary operations $g_1, \dots, g_m \in \clone{A}$;

\item an \emph{h1 clone homomorphism}, or \emph{preserves identities of height $1$}, if 
\[
\xi(f(\pi^m_{i_1}, \dots, \pi^m_{i_n})) = \xi(f)(\pi^m_{i_1}, \dots, \pi^m_{i_n})
\]
for all $n, m\geq 1$, all  $i_1, \dots, i_n \in \{1, \dots, m\}$, and any $n$-ary operation $f\in\cloA$;

\item a \emph{strong h1 clone homomorphism}, or \emph{preserves identities of height at most $1$}, if it  is an h1 clone homomorphism and preserves all  projections.
\end{itemize}
\end{definition}

In the following, we will give some motivation for our terminology.

\begin{definition} \label{def:5.2}
Let $\tau$ be a functional signature, and let $t, s$ be $\tau$-terms. 
An identity $t \approx s$ is said to have \emph{height $1$} (\emph{height at most $1$}, respectively) if both $t$ and $s$ are terms of height $1$ (height at most $1$, respectively). 
\end{definition}
So a height $1$ identity is of the form
\[
f(x_{1},\ldots,x_{n}) \approx g(y_{1},\ldots,y_{m})
\]
where $f,g$ are functional symbols in $\tau$ and $x_1,\dots,x_n, y_1,\dots, y_m$ are not necessarily distinct variables. Identities of height at most $1$ include moreover identities of the form 
\[
f(x_{1},\ldots,x_{n}) \approx y
\]
and 
\[
x \approx y\; .
\]
We note that identities of height at most $1$ are also called \emph{linear} in the literature.

Observe that the variants of a clone homomorphism introduced in Definition~\ref{defn:clonehomo} have the suggested meaning. Indeed, if $\xi\colon \clone{A} \to \clone{B}$ is an arity preserving mapping, $\alg{A}$ is the algebra $(A; (f)_{f\in\cloA})$ of signature $\tau = \clone{A}$, and $\alg{B}$ is the $\tau$-algebra $(B; (\xi(f))_{f\in\cloA})$, then
$\xi$ is a clone homomorphism if and only if $s \approx t$ is true in $\alg{B}$ whenever $s \approx t$ is true in $\alg{A}$, where $s \approx t$ is an identity in the signature $\tau$; similarly, it is an h1 clone homomorphism if and only if this condition holds for identities of height 1, and a strong h1 clone homomorphism if and only if it holds for identities of height at most 1.

\begin{proposition}\label{prop:abstract}
Let $\clone{A}$, $\clone{B}$ be function clones. Then
\begin{itemize}
\item[(i)] $\clone{B} \in \EHSP \clone{A}$ iff there exists a clone homomorphism $\clone{A} \to \clone{B}$;
\item[(ii)] $\clone{B} \in \ERetP \clone{A}$ iff there exists a strong h1 clone homomorphism $\clone{A} \to \clone{B}$;
\item[(iii)] $\clone{B} \in \ETraP \clone{A}$ iff there exists an h1 clone homomorphism $\clone{A} \to \clone{B}$.
\end{itemize}
In all cases, if $A$ and $B$ are finite, then the operator $\PPP$ can be equivalently replaced by $\Pfin$.
\end{proposition}
\begin{proof}
The implications from left to right follow from the fact that the operators $\HHH$, $\SSS$, and $\PPP$ preserve all identities, that the operator $\Ret$ preserves identities of height at most $1$, and that $\Tra$ preserves identities of height $1$.

We show the converses, starting with~(i) although this follows from Birkhoff's theorem.  
Let $\xi\colon \cloA\To\cloB$ be a clone homomorphism. For
every $b\in B$, let $\pi^B_b\in A^{A^B}$ be the function which projects any
tuple in $A^B$ onto the $b$-th coordinate. Let $\cloA$ act on the tuples
$\{\pi^B_b : b\in B\}$ componentwise; closing the latter subset of $A^{A^B}$
under this action, we obtain an invariant subset $S$ of $A^{A^B}$. The action of
$\cloA$ thereon is a function clone in $\SP\cloA$. In fact, if we see this
action as an algebra with signature $\cloA$, it is the free algebra in the
variety generated by the algebra $(A; (f)_{f\in \cloA})$ with generators
$\{\pi^B_b : b\in B\}$. The mapping from $\{\pi^B_b : b\in B\}$ to $B$
which sends every $\pi^B_b$ to $b$ therefore extends to a homomorphism $h\colon
S\To B$ from the free algebra $(S;(f)_{f\in \cloA})$ onto the algebra
$(B;(\xi(f))_{f\in \cloA})$, since the latter algebra is an element of the
mentioned variety as $\xi$ preserves identities. Therefore, the function clone
$\{\xi(f) : {f\in \cloA}\}$ is an element of $\HSP(\cloA)$, and whence
$\clone{B} \in \EHSP \clone{A}$.

With the intention of modifying this proof for (ii) and (iii), let us remark the following. 
 If one wishes to avoid reference to free algebras in the previous proof, then
 it is enough to define the set $S$ as above and then observe that for all $n,
 m\geq 1$, all $n$-ary $f\in\cloA$ and all $m$-ary $g\in\cloA$, and all
 $b_1,\ldots,b_n,c_1,\ldots,c_m\in B$ we have that if
 $f(\pi^B_{b_1},\ldots,\pi^B_{b_n})=g(\pi^B_{c_1},\ldots,\pi^B_{c_m})$, then
 $\xi(f)(b_1,\ldots,b_n)=\xi(g)(c_1,\ldots,c_m)$. This follows from the fact
 that $\xi$ preserves, in particular, identities of height $1$, and allows us to
 define a~mapping $h\colon S\To B$ uniquely by sending every tuple in $S$ of the
 form $f(\pi^B_{b_1},\ldots,\pi^B_{b_n})$ to $\xi(f)(b_1,\ldots,b_n)$. In the case of a clone homomorphism $\xi\colon\cloA\To\cloB$, this yields a homomorphism from the action of $\cloA$ on $S$ onto the action of $\xi[\cloA]$ on $B$; in other words, the action of $\xi[\cloA]$ on $B$ is isomorphic to the action of $\cloA$ on the kernel classes of $h$.

We prove (ii). By the argument above, we obtain a surjective mapping $h\colon
S\To B$ which sends every $\pi^B_b$ to $b$ since $\xi$ preserves projections. Let $h'\colon B\To S$ be the mapping which sends every $b\in B$ to $\pi^B_b$. Then the function clone $\{\xi(f): {f\in \cloA}\}$ is a retraction of the action of $\cloA$ on $S$ via the functions $h'$ and $h$, and so it is an element of $\Ret\SP\cloA$, which equals $\Ret\PPP\cloA$ by Lemma~\ref{lem:algebraicinclusions}. Whence, $\clone{B} \in \ERetP \clone{A}$.

The proof of~(iii) is identical, with the difference that $h$ does not
necessarily send every $\pi^B_b$ to $b$; this is because $\xi$ does not
necessarily preserve projections. Defining $h'$ the same way as above, we then
get that $\{\xi(f) : {f\in \cloA}\}$ is a \transfer{} of the action of $\cloA$
on $S$ via the functions $h'$ and $h$, rather than a retraction.

The additional statement about finite domains follows from the proof above: the power we take is $A^B$.
\end{proof}

We remark that in the proof above, the mapping $h'$ was always injective. Hence, one could alter the definition of a \transfer{} by requiring that the mapping $h_1\colon B\To A$ is injective, and obtain the very same syntactic characterization. In other words, if we introduced an operator $\Tra'$ for such \transfers{}, from which we shall refrain, then we would have $\RP\cloA=\Tra'\PPP\cloA$ for all function clones $\cloA$.

Let us conclude this section with an analogue of Birkhoff's HSP theorem for closure under \transfers{} and products. 

\begin{corollary} \label{cor:linear-birkhoff}
Let $\class{K}$ be a nonempty class of algebras of the same signature $\tau$.
\begin{itemize}
\item[(i)] $\class{K}$ is closed under $\Tra$ and $\PPP$ if and only if $\class{K}$ is the class of models of some set of $\tau$-identities of height $1$.
\item[(ii)] $\class{K}$ is closed under $\Ret$ and $\PPP$ if and only if $\class{K}$ is the class of models of some set of $\tau$-identities of height at most $1$.
\end{itemize}
\end{corollary}
\begin{proof} 
The implications from right to left are trivial since $\PPP$ preserves all identities and since $\Tra$ and $\Ret$ preserve identities of height $1$ and identities of height at most $1$, respectively.

Suppose that $\class{K}$ is closed under $\Tra$ and $\PPP$, and let $\alg{B}$ be any $\tau$-algebra satisfying the set $\Sigma$ of those $\tau$-identities of height $1$ which hold in all algebras of $\class{K}$. Pick for every $\tau$-identity of height $1$ which is not contained in $\Sigma$ an algebra in $\class{K}$ which does not satisfy this identity, and let $\alg{A}\in\class{K}$ be the product of all such algebras. Then the mapping which sends every $\tau$-term of $\alg{A}$ to the corresponding term of $\alg{B}$ preserves identities of height $1$, and so $\alg{B}\in \RP\alg{A}$ by Proposition~\ref{prop:abstract}. Hence, $\alg{B}\in\class{K}$.

The proof of (ii) is identical.
\end{proof}

\section{Topological Linear Birkhoff} \label{sec:cont}

\subsection{Finite goal structures} In Proposition~\ref{prop:abstract} we obtained a characterization when $\clone{B} \in \ETraPfin \clone{A}$ for function clones on finite domains. For function clones on arbitrary sets, even for polymorphism clones of countable $\omega$-categorical structures, we are generally forced to take infinite powers, obstructing the combination with Corollary~\ref{cor:pp_mut}. By considering the topological structure of function clones in addition to their algebraic structure, a characterization of when $\clone{B} \in \HSPfin \clone{A}$ has been obtained for polymorphism clones of $\omega$-categorical structures~\cite{Topo-Birk} -- cf.~Theorem~\ref{thm:old_infinite}. We will now obtain a similar characterization for our operators in the case where $\clone{B}$ has a finite domain. This is, in particular, interesting in the light of Conjecture~\ref{conj:old} which states that for a certain class of $\omega$-categorical structures, the only reason for hardness of the CSP is reduction of the CSP of a structure on a 2-element domain whose only polymorphisms are projections.

In the following proposition, item~(i) is a variant of a statement in~\cite{Topo-Birk} which uses stronger general assumptions, namely, that $\cloA$ is dense in the polymorphism clone of a countable $\omega$-categorical structure; on the other hand, it uses a formally weaker statement in one of the sides of the equivalence, namely, continuity rather than uniform continuity. Continuity and uniform continuity, however, turn out to be equivalent for that setting; cf.~also Section~\ref{sect:cont-unifcont}. Our variant presented here, first observed in~\cite{GPin15}, follows directly from the right interpretation of the proof in~\cite{Topo-Birk}.

We say that the domain of a function clone $\cloB$ is \emph{finitely generated by $\cloB$} iff the algebra obtained by viewing the elements of $\cloB$ as the operations of the algebra is finitely generated; that is, there is a finite subset $B'$ of $B$ such that every element of $B$ is of the form $f(b_1,\ldots,b_n)$, where $b_1,\ldots,b_n\in B'$ and $f\in \cloB$. We remark that domains of polymorphism clones of countable $\omega$-categorical structures are finitely generated by those clones.

\begin{proposition}\label{prop:topological}
Let $\clone{A}, \clone{B}$ be function clones.
\begin{itemize}
\item[(i)] If the domain of $\cloB$ is finitely generated by $\cloB$, then $\clone{B}\in \EHSPfin(\clone{A})$ iff there exists a~uniformly continuous clone homomorphism $\xi\colon \clone{A}\To\clone{B}$.
\item[(ii)] If the domain of $\cloB$ is finite, then $\clone{B}\in \ETraPfin(\clone{A})$ iff there exists a~uniformly continuous h1 clone homomorphism $\xi\colon \clone{A}\To\clone{B}$.
\end{itemize}
\end{proposition}
\begin{proof}
As in Proposition~\ref{prop:abstract}, the implications from left to right are trivial. 

For the other direction, we follow the proof of that proposition, but use uniform continuity to replace arbitrary powers by finite ones. To do this for item~(i),  let $b_1,\ldots,b_n$ be generators of $B$. By uniform continuity, there exist $m\geq 1$ and $a^1,\ldots,a^m\in A^n$ such that for all $n$-ary functions $f,g\in\cloA$ we have that if $f,g$ agree on all tuples $a^1,\ldots,a^m$, then $\xi(f)(b_1,\ldots,b_n)=\xi(g)(b_1,\ldots,b_n)$. For all $1\leq i\leq n$, let $a_i\in A^m$ consist of the $i$-th components of the tuples $a^j$. Then, $f(a_1,\ldots,a_n)=g(a_1,\ldots,a_n)$, calculated componentwise, implies that $\xi(f)(b_1,\ldots,b_n)=\xi(g)(b_1,\ldots,b_n)$. Let $S\subseteq A^m$ be the set obtained by applying $n$-ary functions in $\cloA$ to $a_1,\ldots,a_n$ componentwise. The remainder of the proof is identical with the second proof of item~(i) of Proposition~\ref{prop:abstract}.

For~(ii), we proceed the same way and then as in the corresponding item of Proposition~\ref{prop:abstract}, assuming that $\{b_1,\ldots,b_n\}$ actually equals $B$. The reason why we need the stronger assumption of finiteness is that we cannot use the generating process since $\xi$ does not necessarily preserve identities (in particular, we would be unable to define the mapping $h'$ as in the proof of Proposition~\ref{prop:abstract}). 
\end{proof}

We remark that the condition in item~(i) of Proposition~\ref{prop:topological} that the domain of $\cloB$ be finitely generated is reasonable (and, in general, necessary): if $\cloA$ is a transformation monoid, then we can let it act on arbitrarily many disjoint copies of its domain simultaneously; this new action $\cloB$ will  not be in $\HSPfin(\cloA)$ for reasons of cardinality of the domain if we take enough copies, but $\cloA$ and $\cloB$ will be isomorphic as monoids via a homeomorphism. We can do the same with a function clone which is a transformation monoid in disguise in the sense that all of its functions depend on at most one variable.

Similarly, the finiteness condition in item~(ii) seems to be necessary in general: note, for example, that any mapping between transformation monoids preserves identities of height 1 since that notion of preservation only becomes non-trivial for functions of several variables. We cannot, however, expect endomorphism monoids of completely unrelated structures to be related via the operators $\Tra$ and $\Pfin$.

\subsection{Continuity versus uniform continuity}\label{sect:cont-unifcont} We will now observe conditions ensuring that continuity of a mapping between function clones implies uniform continuity, in particular in order to shed light on the question why uniform continuity appears in Proposition~\ref{prop:topological}, whereas it does not in earlier results such as Theorem~\ref{thm:old_infinite}.

\begin{defn}
Let $\cloA$ be a function clone. An \emph{invertible} of $\cloA$ is a unary
bijection in $\cloA$ whose inverse is also an element of $\cloA$.
\end{defn}

\begin{defn}\label{defn:outerinvertibles}
A mapping $\xi\colon \cloA\To\cloB$ between function clones $\cloA, \cloB$ \emph{preserves composition with invertibles from the outside} iff $\xi(\alpha f)=\xi(\alpha)\xi(f)$ for all invertibles $\alpha\in\cloA$ and all $f\in\cloA$.
\end{defn}

We briefly mentioned the following fact in the discussion preceding Proposition~\ref{prop:topological}; it follows from the material in~\cite{Topo-Birk}, but we give a compact proof here.

\begin{prop}\label{prop:uniform}
Let $\mathscr A,\mathscr B$ be function clones, where $\mathscr A$ is the polymorphism clone of a countable $\omega$-categorical structure. Then any continuous mapping $\xi\colon \mathscr A\To \mathscr B$ preserving composition with invertibles from the outside is uniformly continuous.
\end{prop}
\begin{proof}
Let $k, j\geq 1$ and $w_1,\ldots,w_k\in B^{j}$ be given; we have to show that there exist $m\geq 1$ and $q_1,\ldots,q_k\in A^{m}$ such that $f(q_1,\ldots,q_k)=g(q_1,\ldots,q_k)$ implies $\xi(f)(w_1,\ldots,w_k)=\xi(g)(w_1,\ldots,w_k)$ for all $k$-ary $f,g\in\mathscr A$.

For $k,m\geq1$, $q_1,\ldots,q_k\in A^{m}$, and $q\in A^m$ we write $O_{q_1,\ldots,q_k}^q$ for the basic open set of $\mathscr A$ which consists of those $k$-ary functions $f\in \mathscr A$ for which $f(q_1,\ldots,q_k)=q$. We use the same notation for the basic open sets of $\mathscr B$. We further write $U_{q_1,\ldots,q_k}^q$ for the open set of  those $k$-ary functions $f$ in $\mathscr A$ for which $f(q_1,\ldots,q_k)$ lies in the orbit of $q$ with respect to the action of the invertible elements of $\mathscr A$ on $k$-tuples. 

By continuity, for every $k$-ary $f\in\mathscr A$ there exist $m^f\geq 1$ and
$q_1^f,\ldots,q_k^f, q^f \in A^{m^f}$ such that $g\in
O_{q_1^f,\ldots,q_k^f}^{q^f}$ implies
$\xi(f)(w_1,\ldots,w_k)=\xi(g)(w_1,\ldots,w_k)$, for all $k$-ary $g\in \mathscr
A$. By a standard compactness argument, the space $\mathscr A \cap {A^{A^k}}$ is
covered by finitely many sets of the form $U_{q_1^f,\ldots,q_k^f}^{q^f}$; write
$f_1,\ldots,f_n$ for the functions involved in this covering. Set
$q_1,\ldots,q_k\in A^{m}$ to be the vectors
obtained by gluing everything together, i.e., $q_i$ is obtained by joining the
vectors $q_i^{f_1}$, \dots, $q_i^{f_n}$, and $m$ is the length of the vectors obtained this way.  Suppose
$f(q_1,\ldots,q_k)=g(q_1,\ldots,q_k)$. There exists $h\in\{f_1,\ldots,f_n\}$
such that $f\in U_{q_1^h,\ldots,q_k^h}^{q^h}$. Thus there exists a unary
invertible $\alpha\in\mathscr A$ such that $\alpha(f)\in
O_{q_1^h,\ldots,q_k^h}^{q^h}$. Because $f(q_1,\ldots,q_k)=g(q_1,\ldots,q_k)$, we
also have $\alpha(g)\in O_{q_1^h,\ldots,q_k^h}^{q^h}$. By definition,
$\xi(\alpha(f))(w_1,\ldots,w_k)=\xi(h)(w_1,\ldots,w_k)=\xi(\alpha(g))(w_1,\ldots,w_k)$.
Hence, because $\xi$ preserves composition with invertibles from the outside, 
\begin{align*}
\xi(f)(w_1,\ldots,w_k)&=\xi(\alpha^{-1}\alpha f)(w_1,\ldots,w_k)\\
&=\xi(\alpha^{-1})\xi(\alpha f)(w_1,\ldots,w_k)\\
&=\xi(\alpha^{-1})\xi(\alpha g)(w_1,\ldots,w_k)\\
&=\xi(g)(w_1,\ldots,w_k)\;. \qedhere
\end{align*}
\end{proof}

%\begin{defn}
%Let $\cloA, \cloB$ be function clones. A mapping $\xi\colon \cloA\To\cloB$ \emph{preserves composition with invertibles from outside} iff $\xi(\alpha f)=\xi(\alpha)\xi(f)$ for all invertibles $\alpha\in\cloA$ and all $f\in\cloA$.
%\end{defn}

\begin{defn}\label{defn:outerinvertibles2}
Let $\cloA$ be a function clone, and consider it as an algebra with signature $\cloA$ and domain $A$. An identity over the signature $\cloA$ is \emph{of height 1 modulo outer invertibles} if it is of the form
\[
\alpha f(x_{1},\ldots,x_{n}) \approx \beta g(y_{1},\ldots,y_{m})
\]
where $f,g\in\cloA$, $\alpha,\beta\in\cloA$ are invertibles, and $x_1,\dots,x_n, y_1,\dots, y_m$ are not necessarily distinct variables.
\end{defn}

Clearly, a mapping $\xi\colon\cloA\To\cloB$ preserves identities of height 1 modulo outer invertibles iff it is an h1 clone homomorphism preserving composition with invertibles from the outside.

\begin{cor}\label{cor:outer}
Let $\mathscr A,\mathscr B$ be function clones, where $\mathscr A$ is the polymorphism clone of an $\omega$-categorical structure, and $\cloB$ has a finite domain. If there exists a continuous mapping $\xi\colon\mathscr A\To \mathscr B$ which preserves height 1 identities modulo outer invertibles, then $\mathscr B\in\ETraPfin(\mathscr A)$.
\end{cor}
\begin{proof}
This follows from Propositions~\ref{prop:uniform} and~\ref{prop:topological}.
\end{proof}

Notice that the operator $\Tra$ preserves height 1 identities, but not necessarily identities which are of height 1 modulo outer invertibles, depriving us in Corollary~\ref{cor:outer} of an equivalence similar to the one in Proposition~\ref{prop:topological}.

\subsection{Infinite goal structures}  In infinite domain constraint satisfaction, structures are generally studied relative to a base structure: one usually starts with a  structure $\relstr{A}$ and studies all structures which are first-order definable in $\relstr{A}$, called \emph{reducts} of $\relstr{A}$~\cite{BodPin-Schaefer-both, tcsps-journal, BodDalMarPin, BodMarMot, BPT-decidability-of-definability}. When $\relstr{A}$ is countable and $\omega$-categorical, then this amounts to the study of all structures whose polymorphism clone contains the automorphism group of $\relstr{A}$. So in a sense, one considers function clones \emph{up to  automorphisms of $\relstr{A}$}. In this section, we see that we can make some of our results work for infinite goal structures when we assume that mappings between function clones are compatible with composition with automorphisms. Definitions~\ref{defn:outerinvertibles} and~\ref{defn:outerinvertibles2} pointed in that direction, but as it turns out, we need to compose functions with invertibles from the inside rather than the outside.

%\begin{defn}\label{defn:innerinvertibles}
%A mapping $\xi\colon \cloA\To\cloB$ between function clones $\cloA, \cloB$ \emph{preserves composition with invertibles from the inside} iff $\xi(f(\beta_1(x_1),\ldots,\beta_n(x_n))=\xi(f)(\beta_1(x_1),\ldots,\beta_n(x_n))$ for all invertibles $\beta_1,\ldots,\beta_n\in\cloA$ and all $f\in\cloA$.
%\end{defn}

%\begin{defn}
%An equation in an abstract clone $\C$ is \emph{linear modulo inner invertibles} iff it is of the form
%$$
%\forall x_1,\ldots,x_k.\; f(\beta_1(x_{i_1}),\ldots,\beta_n(x_{i_n}))=g(x_{j_1},\ldots,x_{j_m})
%$$
%where $n,m,k\geq 1$, $\{i_1,\ldots,i_n,j_1,\ldots,j_m\}\subseteq\{1,\ldots,k\}$, 
% $\beta_1,\ldots,\beta_n \in\C$ are unary invertibles, and $f,g\in\C$ are $n$- and $m$-ary, respectively.
%\end{defn}

\begin{defn}\label{defn:innerinvertibles}
Let $\cloA$ be a function clone, and consider it as an algebra with signature $\cloA$ and domain $A$. An identity over the signature $\cloA$ is \emph{of height 1 modulo inner invertibles} if it is of the form
\[
f(\alpha_1(x_{1}),\ldots,\alpha_n(x_{n})) \approx g(\beta_1(y_{1}),\ldots,\beta_m(y_{m}))
\]
where $f,g\in\cloA$, $\alpha_1,\ldots,\alpha_n,\beta_1,\ldots,b_m\in\cloA$ are invertibles, and $x_1,\dots,x_n, y_1,\dots, y_m$ are not necessarily distinct variables.
\end{defn}

\begin{defn}
A~mapping $\xi\colon \cloA \to \cloB$ between function clones \emph{preserves
height 1 identities modulo inner invertibles} iff
$\xi(f(\alpha_1,\dots,\alpha_n)) = \xi(f)(\xi(\alpha_1),\dots,\xi(\alpha_n))$ for
all invertibles $\alpha_1,\dots,\alpha_n \in \cloA$ and all $n$-ary $f\in \cloA$.
\end{defn}

%Similarly as before, a mapping $\xi\colon\cloA\To\cloB$ preserves identities of
%height 1 modulo inner invertibles iff it is an h1 clone homomorphism preserving
%composition with invertibles from the inside.

\begin{prop}\label{prop:inner}
Let $\mathscr A, \mathscr B$ be function clones, and let $\xi\colon \mathscr A\To\mathscr B$
\begin{itemize}
\item be uniformly continuous;
\item preserve height 1 identities modulo inner invertibles;
\item be so that the invertibles of the image $\xi[\mathscr A]$ act with finitely many orbits on $B$.
\end{itemize}
Then $\mathscr B\in \ETraPfin(\mathscr A)$.
\end{prop}
\begin{proof}
Let $d_1,\ldots,d_k\in B$ be representatives of the orbits of the action of the invertibles of $\xi[\mathscr A]$ on $B$. By uniform continuity, there exist $m\geq 1$ and $c_1,\ldots,c_k\in A^m$ such that for all $k$-ary $f,g\in\mathscr A$ we have that $f(c_1,\ldots,c_k)=g(c_1,\ldots,c_k)$ implies $\xi(f)(d_1,\ldots,d_k)=\xi(g)(d_1,\ldots,d_k)$. Define a mapping $h_1\colon A^m\To B$ by setting $h_1(f(c_1,\ldots,c_k)):=\xi(f)(d_1,\ldots,d_k)$ for all elements of the form $f(c_1,\ldots,c_k)$ for some $f\in \mathscr A$, and by extending it arbitrarily to $A^m$. By the choice of $c_1,\ldots,c_k\in A^m$, this mapping is well-defined. Next define $h_2\colon B\To A^m$ as follows: for each $d\in B$, pick an invertible $\alpha\in \mathscr A$ and $1\leq i\leq k$ such that $d=\xi(\alpha)(d_i)$, and set $h_2(d):=\alpha(c_i)$. We claim that for all $k$-ary $f\in\mathscr A$ we have $\xi(f)=h_1(f(h_2(x_1),\ldots,h_2(x_k)))$.  It then follows that $\xi[\mathscr A]$ is a~\transfer{} of the action of $\mathscr A$ on $A^m$, proving the statement.

To see the claim, let $f$ be given, and let $s_1,\ldots,s_k\in B$. Pick invertibles $\alpha_1,\ldots,\alpha_k\in\mathscr A$ and $i_1,\ldots,i_k\in\{1,\ldots,k\}$ such that $\xi(\alpha_j)(d_{i_j})=s_j$ for all $1\leq j\leq k$. Then
\begin{align*}
h_1(f(h_2(s_1),\ldots,h_2(s_k)))&=h_1(f(h_2(\xi(\alpha_1)(d_{i_1})),\ldots,h_2(\xi(\alpha_k)(d_{i_k}))))\\
&=h_1(f(\alpha_1(c_{i_1}),\ldots,\alpha_k(c_{i_k})))\\
&=h_1(f(\alpha_1,\ldots,\alpha_k)(c_{i_1},\ldots,c_{i_k}))\\
&=\xi(f(\alpha_1,\ldots,\alpha_k))(d_{i_1},\ldots,d_{i_k}))\\
&=\xi(f)(\xi(\alpha_1)(d_{i_1}),\ldots,\xi(\alpha_k)(d_{i_k})))\\
&=\xi(f)(s_1,\ldots,s_k).
\end{align*}
Here we use the definitions of $h_1$ and $h_2$ and that $\xi$ preserves height 1
identities modulo inner invertibles.
\end{proof}

Observe that strong h1 clone homomorphisms between function clones which
preserve identities of height 1 modulo inner invertibles send invertibles to
invertibles. In particular, the image of the group of invertibles of a function
clone under such a mapping is a group. 

\begin{cor}
Let $\relstr{A}, \relstr{B}$ at most countable $\omega$-categorical relational structures. Suppose that $\xi\colon \mathscr A\To\mathscr B$ 
\begin{itemize}
\item is uniformly continuous;
\item preserves identities of height 1 modulo inner  invertibles;
\item is so that the invertibles of $\xi[\mathscr A]$ act with finitely many orbits on $B$.
\end{itemize}
Then $\relstr{B}$ is pp-constructible from $\relstr{A}$.
\end{cor}

%Observe that mappings between clones which preserve almost linear equations send unary invertibles to unary invertibles. In particular, the image of the group of invertibles of a clone under such a mapping is a group. WHAT THE HELL? This is super wrong. Delete?

%TODO: one can also consider weak oligomorphicity and linear equations modulo outer unaries or other variants. This might be useful, see e.g. the paper Projective clone homomorphisms, where we obtain equations modulo outer elementary embeddings. At least for elementary embeddings the same proof should work with some local inverse. Not sure about other unaries.

\section{Coloring of clones by relational structures}\label{sec:snek}

%All function clones are naturaly quasi-ordered by an existence of a~clone
%homomorphism between them, i.e., $\cloA \le \cloB$ if there is a~clone
%homomorphism from $\cloA$ to $\cloB$. After factoring out homomorphically
%equivalent clones, one gets a~partially ordered class, that is in fact lattice
%ordered. This class-size lattice is usually called the lattice of
%interpretability types of varieties since the same `lattice' can be also
%constructed by starting with ordering of varieties by an~existence of
%interpretability. It has been thoroughly studied in
%a~monograph~\cite{garcia.taylor84}. If we start with a~Mal'cev condition and
%consider all varieties satisfying this conditions (or all function clones
%containing functions satisfying the desired identities), we get a~filter in this
%lattice.  In this way every Mal'cev condition can be interpreted as a~filter,
%with the strong Mal'cev conditions corresponding to certain principal filters.
%One notable still opened conjecture is that the filter corresponding to Day
%terms that describe congruence modular varieties (see \cite{day69}) is
%oin-prime.  Partial results on this conjecture has been published in
%\cite{sequeira01} and \cite{bentz.sequeira14}.

In order to define the central concept, we first introduce some notation, a piece of which has already appeared in the proof of  Proposition~\ref{prop:abstract}.

Let $\cloA$ be a clone and $B$ a set. 
For an element  $b\in B$, let $\pi^B_b\in A^{A^B}$ be the function which projects every
tuple in $A^B$ onto the $b$-th coordinate,  
 let $F_{\cloA}(B) \subseteq A^{A^B}$  be the closure of  $\{\pi^B_b : b\in B\}$ under the componentwise action of~$\cloA$,
and let $\clone{F}_{\cloA}(B)$ be the corresponding clone acting on $F_{\cloA}(B)$.  (We mentioned in the proof of Proposition~\ref{prop:abstract} that $F_{\cloA}(B)$, denoted $S$ there, is the universe of the free algebra generated by the algebra $(A; (f)_{f\in \cloA})$ with generators
$\{\pi^B_b : b\in B\}$.) 
Note that for a finite $B = \{0,1, \dots, n-1\}$, $F_{\cloA}(B)$ is equal to the set of $n$-ary operations in $\cloA$.

For a~relation $R \subseteq B^k$ we define a~relation $R^\cloA
\subseteq F_{\cloA}(B)^k$ as the closure of the set
\[
\{ (\pi_{b_1}^B, \dots, \pi_{b_k}^B) : (b_1,\dots,b_k) \in R\}
\]
under the componentwise action of~$\clone{F}_{\cloA}(B)$.  
In this way, each relational structure $\relB$ with universe $B$ lifts to a relational structure with universe $F_{\cloA}(B)$ of the same signature. Colorings are defined as homomorphisms from the lifted structure to $\relstr{B}$:

\begin{definition}
Let $\cloA$ be a function clone and let $\relB$ be a relational structure.
We say that a~mapping $c\colon F_{\cloA}(B) \to B$ is a~\emph{coloring of $\cloA$ by
$\relB$} if for all relations $R$ of $\relB$ and all tuples \( (f_1,\dots,f_k)
\in R^\cloA \) we have \( (c(f_1),\dots,c(f_k)) \in R \). A~\emph{strong
coloring} is a~coloring that in addition satisfies $c(\pi_b^B) = b$ for all $b
\in B$.  A~clone is (strongly) \emph{$\relB$-colorable} if there exists a~(strong)
coloring $c$ of the clone by $\relB$.
\end{definition}

Sequeira's notion of compatibility with projections~\cite{sequeira01} is equivalent to strong colorings by relational structures consisting of equivalence relations. 

The proof of the following proposition illustrates how a specific Maltsev condition for $n$-permutability is translated to strong non-colorability. 

\begin{proposition} \label{prop:hm}
A variety $\var{V}$ is congruence $n$-permutable for some $n$ if and only if $\varclo{\var{V}}$ is not strongly $(\{0,1\}; \leq)$-colorable.
\end{proposition}

\begin{proof}
Hagemann-Mitschke operations are ternary operations $p_1$, \dots, $p_{n-1}$
such that
\(
  p_1(x,y,y) \approx x \), \(
    p_{n-1}(x,x,y) \approx y\), and \(
  p_i(x,x,y) \approx p_{i+1}(x,y,y)
\) for every $i=1,\dots,n-2$.
By~\cite{hagemann.mitschke73}, a variety $\var{V}$ is $n$-permutable if and only if $\cloA = \varclo{\var{V}}$ contains Hagemann-Mitschke operations.

The relation $\leq^{\cloA}$ is the closure of $\{(\pi^{\{0,1\}}_0,\pi^{\{0,1\}}_0), (\pi^{\{0,1\}}_0,\pi^{\{0,1\}}_1), (\pi^{\{0,1\}}_1,\pi^{\{0,1\}}_1)\}$ under the componentwise action of $\cloA$, therefore it is equal to
$$
\{(t(\pi^{\{0,1\}}_0,\pi^{\{0,1\}}_0,\pi^{\{0,1\}}_1), t(\pi^{\{0,1\}}_0,\pi^{\{0,1\}}_1,\pi^{\{0,1\}}_1)): t \in \cloA, \mbox{ $t$ is ternary }\}.
$$
In other words, for two binary operations $f,g$ in $\cloA$ we have
$f \leq^\cloA g$ if and only if there exists a~ternary
operation $t\in \cloA$ satisfying $t(x,x,y) \approx f(x,y)$ and $t(x,y,y) \approx
g(x,y)$. It follows that if a~clone contains
Hagemann-Mitschke operations then it is not strongly $(\{0,1\};\leq)$-colorable
since such operations force $c(\pi_1^{\{0,1\}}) \le c(\pi_0^{\{0,1\}})$,
a~contradiction with $c(\pi_i^{\{0,1\}}) = i$.  For the other implication, we
 define a~strong coloring $c$ by 
$c(h) = 0$ iff there exists a~Hagemann-Mitschke
chain connecting $\pi^{\{0,1\}}_0$ and $h$, i.e., there is $n$, and operations $p_1,\dots,p_n \in \cloA$ such
that
\(
  p_1(x,y,y) \approx x \), \(
    p_n(x,x,y) \approx h(x,y) \), and \(
  p_i(x,x,y) \approx p_{i+1}(x,y,y)
\)
for every $i=1,\dots,n-1$. Since $\cloA$ does not contain Hagemann-Mitschke operations, we get that
$c(\pi_1^{\{0,1\}}) = 1$. The rest is an~easy exercise.
\end{proof}

A similar characterization of congruence modularity follows from~\cite{sequeira01}:

\begin{proposition} \label{prop:day}
A variety $\var{V}$ is congruence modular if and only if $\varclo{\var{V}}$ is not strongly $\relstr{D}$-colorable, where
$\relstr{D} = (D; \alpha, \beta, \gamma)$, $D = \{1,2,3,4\}$, and $\alpha$, $\beta$, $\gamma$ are equivalence relations on $D$ defined by partitions $12|34$, $13|24$, $12|3|4$. 
\end{proposition}

The connection between colorings of clones by relational structures and
h1 clone homomorphisms is presented in the following proposition.

\begin{proposition}\label{prop:coloring-and-h1}
Let $\cloA$ be a~function clone and let $\relB$ be a~relational structure. 
\begin{enumerate}
\item[(i)] $\cloA$ is $\relB$-colorable if and only if there is
an~h1 clone homomorphism $\xi\colon \cloA \to \cloB$, and 
\item[(ii)] $\cloA$ is strongly $\relB$-colorable if and only if
there is a~strong h1 clone homomorphism $\xi\colon \cloA \to \cloB$.
\end{enumerate}
\end{proposition}

\begin{proof}
The proof relies on Proposition \ref{prop:abstract} and its proof.

To prove (i), suppose that $c\colon F_{\cloA}(B) \to B$ is a~coloring of
$\cloA$ by $\relB$ and consider the \transfer{} $\clone{C} = \{f': f\in\cloA\}$ of $\clone{F}_{\cloA}(B)$ given by $c$ and the mapping $b \mapsto \pi_b^B$ (thus the clone $\clone{C}$ acts on $B$).
We claim that each $f'$ is a polymorphism of $\relstr{B}$. To verify this, consider a~relation $R$ of $\relB$ and tuples
$(b_{i1},\dots,b_{ik})\in R$, $i=1,\dots,n$. We have
\(
  \big(
    f(\pi_{b_{11}}^B,\dots,\pi_{b_{n1}}^B), \dots,
      f(\pi_{b_{1k}}^B,\dots,\pi_{b_{nk}}^B) 
  \big) \in R^\cloA
\),
and then
\[
  \big(
    f'(b_{11},\dots,b_{n1}), \dots, f'(b_{1k},\dots,b_{nk}) 
  \big)
  =
  \big(
    c(f(\pi_{b_{11}}^B,\dots,\pi_{b_{n1}}^B)), \dots,
      c(f(\pi_{b_{1k}}^B,\dots,\pi_{b_{nk}}^B)) 
  \big) \in R
\]
from the definition of coloring. We have shown that $\clone{C} \subseteq \cloB$, therefore $\cloB \in \ETraP \cloA$ since $\clone{C} \in \Tra \SSS \PPP \cloA = \Tra \PPP \cloA$. From Proposition~\ref{prop:abstract} it now follows that there exists an~h1 clone homomorphism $\cloA \to \cloB$.

For the other implication suppose that we have an~h1 clone homomophism $\xi$ from
$\cloA$ to $\cloB$. Then, from the proof of Proposition~\ref{prop:abstract},
we know that $\cloB$ is an expansion of the \transfer{} of~$\clone{F}_{\cloA}(B)$
given by the mappings $h = c \colon f(\pi_{b_1}^B,\dots,\pi_{b_n}^B)
\mapsto \xi(f)(b_1,\dots,b_n)$ and $h'\colon b \mapsto \pi_b^B$. We will show that $c$ is a~coloring.
Let $R$ be a~relation in $\relB$ and 
$(f_1,\dots,f_k) \in R^\cloA$. By the definition of $R^{\cloA}$, there exists $f\in \cloA$ and tuples
$(b_{11},\dots,b_{1k}), \dots, (b_{n1},\dots,b_{nk}) \in R$ such that $f_i =
f(\pi_{b_{1i}}^B,\dots,\pi_{b_{ni}}^B)$ for all $i=1,\dots,n$.  Therefore,
\begin{multline*}
  \big(
    c(f_1), \dots, c(f_k) 
  \big) =
  \big(
    c(f(\pi_{b_{11}}^B,\dots,\pi_{b_{n1}}^B)), \dots,
      c(f(\pi_{b_{1k}}^B,\dots,\pi_{b_{nk}}^B)) 
  \big) \\
  = \big(
    \xi(f)(b_{11},\dots,b_{n1}), \dots, \xi(f)(b_{1k},\dots,b_{nk})
  \big) \in R
\end{multline*}
since $\xi(f)$ is compatible with $R$. This concludes the proof of (i); the proof
of (ii) is analogous.
\end{proof}

%As a~corollary of this claim and the previous lemma, we get that every clone
%either constains Hagemann-Mitschke functions, or has a~strong h1-homomorphism to
%polymorphism clone of $(\{0,1\},\leq)$. Note that a~similar result is known for
%dempotepotent clones, an idempotent clone $\cloA$ either contains
%Hagemann-Mitschke functions; or it has a~clone homomorphism to the polymorphism
%clone of $(\{0,1\},\leq)$ (see \cite{valeriote.willard14}). 

%We have a~similar situation for several other Mal'cev conditions. That is
%a~clone does not contain the corresponding function(s) if and only if it is not
%(strongly) colorable by some relational structure.  We present such relational
%structure for Day terms that has been described among others in \cite{sequeira01}.

%Immediately from Lemma \ref{lem:coloring-and-h1}, we get that a~function clone
%contains functions satisfying Day identities if and only if there is no strong
%h1 clone homomorphism to the polymorphism clone of $\rel D$.

As a corollary of the last three propositions we get that a variety is congruence $n$-permutable for some $n$ (modular, respectively) if and only if its clone does not have a~strong h1 clone homomorphism to the polymorphism clone of $(\{0,1\}; \leq)$ (the structure $\relstr{D}$ from Proposition~\ref{prop:day}, respectively). 

%By combining Proposition~\ref{prop:coloring-and-h1} and Proposition~\ref{prop:abstract} we get the following consequence concerning the colorability of joins.

%Given that strong h1 clone homomorphisms encode the identities of height at
%most~1 that are satisfied in the clone, we can obtain that the modularity
%conjecture is true for varieties defined by identities of height at most~1. This
%generalizes result of \cite{bentz.sequeira14}.
%The similar result can be obtained for every Mal'cev condition that can be
%described in similar manner by (non-existence) of some coloring by a~relational
%structure.

The application of colorings formulated as Theorem~\ref{thm:perm-and-modular} is based on the following 
consequence of Proposition \ref{prop:abstract} and Corollary \ref{cor:linear-birkhoff}. 

\begin{proposition} Let $\var V$ be a~variety and $\cloB$ be a~function clone.
\begin{enumerate}
\item[(i)] If $\var V$ is defined by identities of height 1 and there is
an~h1 clone homomorphism from $\varclo{\var V}$ to $\cloB$, then there is also
a~clone homomorphism from $\varclo{\var V}$ to $\cloB$. 
\item[(ii)] If $\var V$ is defined by identities of height at most 1 and there is
a~strong h1 clone homomorphism from $\varclo{\var V}$ to $\cloB$, then there is also
a~clone homomorphism from $\varclo{\var V}$ to~$\cloB$.
\end{enumerate}
\end{proposition}

\begin{proof}
We prove the first part, the second part is analogous. 
Since there is an~h1 clone homomorphism from $\varclo{\var V}$ to $\cloB$, we have $\cloB \in \ETraP\varclo{\var V}$,
therefore $\cloB \in \ETraP \algA$ for a generator $\alg A$ of $\var V$.
But $\var V$ is defined by identities of height 1, so it is closed under $\Tra$ and $\PPP$, hence $\cloB \in \EEE \var{V}$, which in turn implies $\cloB \in \EHSP \varclo{\var V}$. It follows that there exists a clone homomorphism from $\varclo{\var V}$ to $\cloB$. 
%
%
% For (i), observe that if $\cloA$ is an
% image of $\varclo{\var V}$ in an~h1 clone homomorphism then 
% $\cloB \in \ETraP\varclo{\var V}$, therefore $\cloB$.
%
%
% But since $\varV$ is defined by identities of
% height 1 it is closed under $\Tra$ and $\PPP$, hence $\cloA \in \Exp \varV$.
% Which means from the Birkhoff theorem that $\cloA \in \EHSP \varclo \varV$, and
% consequently there is a~clone homomorphism from $\varclo \varV$ to $\cloA$.
% The proof of (ii) is analogous.
\end{proof}

A combination of this proposition and Proposition~\ref{prop:coloring-and-h1} immediately gives the following.

\begin{corollary} \label{cor:join}
Let $\var V$ and $\var W$ be varieties and let $\relstr{B}$ be a relational structure.
\begin{itemize}
\item[(i)] If $\var V$ and $\var W$ are defined by identities of height $1$, and  $\varclo{\var{V}}$ and $\varclo{\var{W}}$ are $\relstr{B}$-colorable, then so is $\varclo{\var{V} \vee \var{W}}$. 
\item[(ii)] If $\var V$ and $\var W$ are defined by identities of height at most $1$, and  $\varclo{\var{V}}$ and $\varclo{\var{W}}$ are strongly $\relstr{B}$-colorable, then so is $\varclo{\var{V} \vee \var{W}}$. 
\end{itemize}
\end{corollary}

\section{Back to the Introduction} \label{sec:wrapup}

Our results imply Theorem~\ref{thm:new_infinite} as follows; note that Theorem~\ref{thm:new_finite} is a special case thereof. The two statements in~(i) are equivalent by Corollary~\ref{cor:relational}. They are equivalent to~(ii) by Corollary~\ref{cor:pp_mut}, and to~(iii) by Proposition~\ref{prop:topological}.

Item (i) in Theorem~\ref{thm:equiv_conditions} is trivially implied by item
(ii). The other direction follows from Taylor's theorem~\cite{T77} which implies
that the non-existence of a clone homomorphism from an idempotent clone $\cloB$
to the clone of projections is witnessed by non-trivial height 1 identities
satisfied by operations in $\cloB$. These identities prevent an h1 clone
homomorphism to the projection clone $\mathbf 1$. Items (i) and (ii) are
equivalent to (iii) by~\cite{BK12} and to (iv) by~\cite{KMM14} (which is a
strengthening of~\cite{Sig10}). The following corollary implies that items
(i)--(iv) are equivalent to their primed versions.

\begin{corollary} \label{cor:weak_homo_to_core}
Let $\relstr{A}$ be at most countable $\omega$-categorical structure and let
$\relstr{B}$ be the model-complete core of $\relstr{A}$ expanded by finitely
many singleton unary relations. Then there exist uniformly continuous h1 clone
homomorphisms $\cloA \to \cloB$ and $\cloB \to \cloA$.
\end{corollary}

\begin{proof}
Since $\relstr{B}$ is pp-constructible from $\relstr{A}$ and $\relstr{A}$ is pp-constructible from $\relstr{B}$, the claim follows from Theorem~\ref{thm:new_infinite}. 
\end{proof}

\begin{corollary}
The old Conjecture~\ref{conj:old} implies the new Conjecture~\ref{conj:new}.
\end{corollary}
\begin{proof}
If the first item of Conjecture~\ref{conj:old} holds for a structure, then so does the first item of Conjecture~\ref{conj:new}. Indeed, let $\relstr{C}$ be an expansion of the model-complete core $\relstr{B}$ of an $\omega$-categorical structure $\relstr{A}$ such that $\clone{C}$ has a continuous homomorphism $\xi$ to the projection clone $\mathbf 1$. Then $\xi$ is uniformly continuous by Proposition~\ref{prop:uniform}.   
By composing $\xi$ with a uniformly continuous h1 clone homomorphism from $\cloA$ into $\cloC$, provided by Corollary~\ref{cor:weak_homo_to_core}, we obtain a~uniformly continuous h1 clone homomorphism from $\cloA$ to $\mathbf 1$.
\end{proof}

We do not know whether the converse holds, i.e., whether or not the first item of Conjecture~\ref{conj:new} implies the first item of Conjecture~\ref{conj:old}. This problem also provides a possible approach to disproving Conjecture~\ref{conj:old}. 

\begin{prob}
Let $\relA$ be a reduct of a finitely bounded homogeneous structure, and let $\relB$ be its model-complete core. Suppose that $\cloA$ maps to $\mathbf 1$ via a~uniformly continuous h1 clone homomorphism (and hence, $\CSP(\relA)$ is NP-hard).\\
 Do there exist elements $b_1,\ldots,b_n$ in $\relB$ such that the polymorphism clone of the expansion of $\relB$ by those constants maps homomorphically and continuously to $\mathbf 1$?
\end{prob}  

%We remark that the answer to this question is positive in the case of a finite structure $\relA$. 
%Note that the above problem provides 

Let us discuss the results on colorings in Section~\ref{sect:intro_colorings}. Theorem~\ref{thm:coloring-and-h1} is the first part of Proposition~\ref{prop:coloring-and-h1}, and Theorem~\ref{thm:perm-and-modular} is a consequence of Corollary~\ref{cor:join}, Proposition~\ref{prop:hm}, and Proposition~\ref{prop:day}. 

The correspondence between Maltsev conditions and h1 clone homomorphism suggests an approach to 
Conjecture~\ref{conj:taylor} and similar problems: for a given clone $\clone B$ (which characterizes the Maltsev condition in question via h1 clone homomorphisms), find an upward directed class of clones such that, for every $\cloA$,  $\cloA$ has an h1 clone homomorphism to $\clone B$ if and only if $\clone A$ has a clone homomorphism to some member of the class.
Encouraging results in this direction are~\cite{kearnes.tschantz07} and \cite{valeriote.willard14}, 
where such a class was found for idempotent $\clone A$s and the $\clone B$s characterizing congruence permutability~\cite{kearnes.tschantz07} and $n$-permutability~\cite{valeriote.willard14}. Is it possible that such a class exists even for every clone $\clone B$?

\bibliographystyle{plain}

\bibliography{CSPbib,snek,global}

\def\cprime{$'$}
\begin{thebibliography}{10}

\bibitem{BSL:9956673}
Libor Barto.
\newblock The constraint satisfaction problem and universal algebra.
\newblock {\em The Bulletin of Symbolic Logic}, 21:319--337, 2015.

\bibitem{BK12}
Libor Barto and Marcin Kozik.
\newblock Absorbing subalgebras, cyclic terms, and the constraint satisfaction
  problem.
\newblock {\em Logical Methods in Computer Science}, 8(1), 2012.

\bibitem{BK14}
Libor Barto and Marcin Kozik.
\newblock Constraint satisfaction problems solvable by local consistency
  methods.
\newblock {\em Journal of the ACM}, 61(1):\#{3}, 1--19, 2014.

\bibitem{bentz.sequeira14}
Wolfram Bentz and Lu{\'{\i}}s Sequeira.
\newblock Taylor's modularity conjecture holds for linear idempotent varieties.
\newblock {\em Algebra Universalis}, 71(2):101--107, 2014.

\bibitem{Berg11}
Clifford Bergman.
\newblock {\em Universal Algebra: Fundamentals and Selected Topics}.
\newblock Pure and Applied Mathematics. Taylor and Francis, 2011.

\bibitem{Bir-On-the-structure}
Garrett Birkhoff.
\newblock On the structure of abstract algebras.
\newblock {\em Mathematical Proceedings of the Cambridge Philosophical
  Society}, 31(4):433--454, 1935.

\bibitem{Cores-journal}
Manuel Bodirsky.
\newblock Cores of countably categorical structures.
\newblock {\em Logical Methods in Computer Science}, 3(1):1--16, 2007.

\bibitem{BodirskySurvey}
Manuel Bodirsky.
\newblock Constraint satisfaction problems with infinite templates.
\newblock In Heribert Vollmer, editor, {\em Complexity of Constraints (a
  collection of survey articles)}, volume 5250 of {\em Lecture Notes in
  Computer Science}, pages 196--228. Springer, 2008.

\bibitem{Bodirsky-HDR}
Manuel Bodirsky.
\newblock Complexity classification in infinite-domain constraint satisfaction.
\newblock M\'emoire d'habilitation \`a diriger des recherches, Universit\'{e}
  Diderot -- Paris 7. Available at arXiv:1201.0856, 2012.

\bibitem{BodDalMarPin}
Manuel Bodirsky, V\'{\i}ctor Dalmau, Barnaby Martin, and Michael Pinsker.
\newblock Distance constraint satisfaction problems.
\newblock In Petr Hlinen\'y and Anton\'in Kucera, editors, {\em Proceedings of
  Mathematical Foundations of Computer Science}, pages 162--173. Springer
  Verlag, August 2010.

\bibitem{tcsps-journal}
Manuel Bodirsky and Jan K\'ara.
\newblock The complexity of temporal constraint satisfaction problems.
\newblock {\em Journal of the ACM}, 57(2):1--41, 2009.
\newblock An extended abstract appeared in the Proceedings of the Symposium on
  Theory of Computing (STOC'08).

\bibitem{BodMarMot}
Manuel Bodirsky, Barnaby Martin, and Antoine Mottet.
\newblock Constraint satisfaction problems over the integers with successor.
\newblock In {\em Proceedings of ICALP}, pages 256--267, 2015.
\newblock ArXiv:1503.08572.

\bibitem{BodirskyNesetrilJLC}
Manuel Bodirsky and Jaroslav Ne\v{s}et\v{r}il.
\newblock Constraint satisfaction with countable homogeneous templates.
\newblock {\em Journal of Logic and Computation}, 16(3):359--373, 2006.

\bibitem{BodPin-Schaefer-both}
Manuel Bodirsky and Michael Pinsker.
\newblock Schaefer's theorem for graphs.
\newblock {\em Journal of the ACM}, 62(3):\#19, 1--52.
\newblock A conference version appeared in the Proceedings of STOC 2011, pages
  655--664.

\bibitem{BP-reductsRamsey}
Manuel Bodirsky and Michael Pinsker.
\newblock Reducts of {R}amsey structures.
\newblock {\em AMS Contemporary Mathematics, vol. 558 (Model Theoretic Methods
  in Finite Combinatorics)}, pages 489--519, 2011.

\bibitem{Topo-Birk}
Manuel Bodirsky and Michael Pinsker.
\newblock Topological {B}irkhoff.
\newblock {\em Transactions of the American Mathematical Society},
  367:2527--2549, 2015.

\bibitem{BPP-projective-homomorphisms}
Manuel Bodirsky, Michael Pinsker, and Andr\'{a}s Pongr\'acz.
\newblock Projective clone homomorphisms.
\newblock {\em Journal of Symbolic Logic}.
\newblock To appear. Preprint arXiv:1409.4601.

\bibitem{Reconstruction}
Manuel Bodirsky, Michael Pinsker, and Andr\'{a}s Pongr\'acz.
\newblock Reconstructing the topology of clones.
\newblock {\em Transactions of the American Mathematical Society}.
\newblock To appear. Preprint available from arXiv:1312.7699.

\bibitem{BPT-decidability-of-definability}
Manuel Bodirsky, Michael Pinsker, and Todor Tsankov.
\newblock Decidability of definability.
\newblock {\em Journal of Symbolic Logic}, 78(4):1036--1054, 2013.
\newblock A conference version appeared in the Proceedings of LICS 2011.

\bibitem{JBK}
Andrei~A. Bulatov, Andrei~A. Krokhin, and Peter~G. Jeavons.
\newblock Classifying the complexity of constraints using finite algebras.
\newblock {\em SIAM Journal on Computing}, 34:720--742, 2005.

\bibitem{BS81}
Stanley~N. Burris and H.~P. Sankappanavar.
\newblock {\em A course in universal algebra}, volume~78 of {\em Graduate Texts
  in Mathematics}.
\newblock Springer-Verlag, New York, 1981.

\bibitem{FV98}
Tom\'{a}s Feder and Moshe~Y. Vardi.
\newblock The computational structure of monotone monadic snp and constraint
  satisfaction: A study through datalog and group theory.
\newblock {\em SIAM Journal on Computing}, 28(1):57--104, 1998.

\bibitem{garcia.taylor84}
Octavio~Carlos Garc{\'{\i}}a and Walter Taylor.
\newblock The lattice of interpretability types of varieties.
\newblock {\em Memoirs of the Americal Mathematical Society}, 50:v+125, 1984.

\bibitem{GPin15}
Mai Gehrke and Michael Pinsker.
\newblock Uniform {B}irkhoff.
\newblock Preprint, 2016.

\bibitem{hagemann.mitschke73}
Joachim Hagemann and A.~Mitschke.
\newblock On {$n$}-permutable congruences.
\newblock {\em Algebra Universalis}, 3:8--12, 1973.

\bibitem{Hodges}
Wilfrid Hodges.
\newblock {\em A shorter model theory}.
\newblock Cambridge University Press, Cambridge, 1997.

\bibitem{KMM14}
Keith~A. Kearnes, Petar Markovi\'c, and Ralph McKenzie.
\newblock Optimal strong {Mal’cev} conditions for omitting type 1 in locally
  finite varieties.
\newblock {\em Algebra universalis}, 72(1):91--100, 2014.

\bibitem{kearnes.tschantz07}
Keith~A. Kearnes and Steven~T. Tschantz.
\newblock Automorphism groups of squares and of free algebras.
\newblock {\em International Journal of Algebra and Computation},
  17(3):461--505, 2007.

\bibitem{LZ07}
Benoit Larose and L{\'a}szl{\'o} Z{\'a}dori.
\newblock Bounded width problems and algebras.
\newblock {\em Algebra Universalis}, 56(3-4):439--466, 2007.

\bibitem{neumann74}
Walter~D. Neumann.
\newblock On {Malcev} conditions.
\newblock {\em Journal of the Australian Mathematical Society}, 17:376--384, 5
  1974.

\bibitem{Pin15}
Michael Pinsker.
\newblock Algebraic and model theoretic methods in constraint satisfaction.
\newblock arXiv:1507.00931, 2015.

\bibitem{sequeira01}
Lu\'is Sequeira.
\newblock {\em {M}altsev Filters}.
\newblock PhD thesis, University of Lisbon, Portugal, 2001.

\bibitem{Sig10}
Mark~H. Siggers.
\newblock A strong {Mal’cev} condition for locally finite varieties omitting
  the unary type.
\newblock {\em Algebra universalis}, 64(1-2):15--20, 2010.

\bibitem{Szendrei}
\'Agnes Szendrei.
\newblock {\em Clones in universal algebra}.
\newblock S\'eminaire de Math\'ematiques Sup\'erieures. Les Presses de
  l'Universit\'e de {M}ontr\'eal, 1986.

\bibitem{T77}
Walter Taylor.
\newblock Varieties obeying homotopy laws.
\newblock {\em Canad. J. Math.}, 29(3):498--527, 1977.

\bibitem{tschantz}
Steven~T. Tschantz.
\newblock Congruence permutability is join prime.
\newblock unpublished, 1996.

\bibitem{valeriote.willard14}
Matthew~A. Valeriote and Ross Willard.
\newblock Idempotent {$n$}-permutable varieties.
\newblock {\em Bulletin of the London Mathematical Society}, 46(4):870--880,
  2014.

\end{thebibliography}

\end{document}